\RequirePackage{fix-cm}
\documentclass[oneside,english,british]{amsart}
\usepackage{lmodern}
\usepackage[T1]{fontenc}
\usepackage[latin9]{inputenc}
\usepackage{geometry}
\usepackage{color}
\usepackage{prettyref}
\usepackage{amsthm}
\usepackage{amstext}
\usepackage{amssymb}

\makeatletter
\numberwithin{equation}{section}
\numberwithin{figure}{section}
\usepackage{enumitem}		
\theoremstyle{plain}
\newtheorem{thm}{\protect\theoremname}[section]
  \theoremstyle{definition}
  \newtheorem{problem}[thm]{\protect\problemname}
  \theoremstyle{remark}
  \newtheorem{rem}[thm]{\protect\remarkname}
  \theoremstyle{plain}
  \newtheorem{lem}[thm]{\protect\lemmaname}
  \theoremstyle{plain}
  \newtheorem{cor}[thm]{\protect\corollaryname}
  \theoremstyle{plain}
  \newtheorem{prop}[thm]{\protect\propositionname}
  \newcounter{casectr}
  \newenvironment{caseenv}
  {\begin{list}{{\itshape\ \protect\casename} \arabic{casectr}.}{%
   \setlength{\leftmargin}{\labelwidth}
   \addtolength{\leftmargin}{\parskip}
   \setlength{\itemindent}{\listparindent}
   \setlength{\itemsep}{\medskipamount}
   \setlength{\topsep}{\itemsep}}
   \setcounter{casectr}{0}
   \usecounter{casectr}}
  {\end{list}}

\makeatother

\usepackage{babel}
  \addto\captionsbritish{\renewcommand{\casename}{Case}}
  \addto\captionsbritish{\renewcommand{\corollaryname}{Corollary}}
  \addto\captionsbritish{\renewcommand{\lemmaname}{Lemma}}
  \addto\captionsbritish{\renewcommand{\problemname}{Problem}}
  \addto\captionsbritish{\renewcommand{\propositionname}{Proposition}}
  \addto\captionsbritish{\renewcommand{\remarkname}{Remark}}
  \addto\captionsbritish{\renewcommand{\theoremname}{Theorem}}
  \addto\captionsenglish{\renewcommand{\casename}{Case}}
  \addto\captionsenglish{\renewcommand{\corollaryname}{Corollary}}
  \addto\captionsenglish{\renewcommand{\lemmaname}{Lemma}}
  \addto\captionsenglish{\renewcommand{\problemname}{Problem}}
  \addto\captionsenglish{\renewcommand{\propositionname}{Proposition}}
  \addto\captionsenglish{\renewcommand{\remarkname}{Remark}}
  \addto\captionsenglish{\renewcommand{\theoremname}{Theorem}}
  \providecommand{\casename}{Case}
  \providecommand{\corollaryname}{Corollary}
  \providecommand{\lemmaname}{Lemma}
  \providecommand{\problemname}{Problem}
  \providecommand{\propositionname}{Proposition}
  \providecommand{\remarkname}{Remark}
\providecommand{\theoremname}{Theorem}

\begin{document}

\title{Density of values of linear maps on quadratic surfaces.}

\author{Oliver Sargent}

\address{\textsc{Department Of Mathematics, University Walk, Bristol, BS8
1TW, UK.}}

\address{\texttt{Oliver.Sargent@bris.ac.uk}}
\begin{abstract}
In this paper we investigate the distribution of the set of values
of a linear map at integer points on a quadratic surface. In particular
we show that this set is dense in the range of the linear map subject
to certain algebraic conditions on the linear map and the quadratic
form that defines the surface. The proof uses Ratner's Theorem on
orbit closures of unipotent subgroups acting on homogeneous spaces. 
\end{abstract}
\maketitle

\section{Introduction }

We are motivated by the following general problem. 
\begin{problem}
\label{problem 1}If $X$ is some rational surface in $\mathbb{R}^{d}$
and $P:X\rightarrow\mathbb{R}^{s}$ is a polynomial map, then what
can one say about the distribution of the set $\left\{ P\left(x\right):x\in X\cap\mathbb{Z}^{d}\right\} $
in $\mathbb{R}^{s}$? 
\end{problem}
One expects to be able to answer Problem \ref{problem 1} by showing
that the set $\left\{ P\left(x\right):x\in X\cap\mathbb{Z}^{d}\right\} $
is dense in $\mathbb{R}^{s}$ under certain dimension and rationality
conditions imposed on $P$. In full generality problem \ref{problem 1}
is unapproachable via available techniques and what is known is limited
to special cases. For instance, when $X=\mathbb{R}^{d}$ Problem \ref{problem 1}
has been considered for $P$ a linear or quadratic map, or combinations
of both. The case when $P$ is linear is classical, and is treated
by Theorem 1 on page 64 of \cite{MR0349591}. When $P$ is quadratic,
Problem \ref{problem 1} is known as the Oppenheim conjecture, density
was first established by the work of G. Margulis in \cite{MR993328}
and subsequently refined by S.G. Dani and G. Margulis in \cite{MR1016271}.
Moreover, in this situation qualitative results have also been established;
initially by S.G. Dani and G. Margulis in \cite{MR1237827} and later
by A. Eskin, S. Mozes and G. Margulis in \cite{MR1609447}. For $P$,
a pair, consisting of a quadratic and linear form, Problem \ref{problem 1}
has been considered by A. Gorodnik in \cite{MR2067128}. The case
when $P$ consists of a system of many linear forms and a quadratic
form has been considered by S.G. Dani in \cite{MR2366232}. A. Gorodnik
also considered the case when $P$ consists of a system of quadratic
forms in \cite{MR2054841}. To the authors knowledge the case when
$X\neq\mathbb{R}^{d}$ has not been considered. The main result of
this paper deals with a case of Problem \ref{problem 1} when $X$
is a quadratic surface and is stated below. 
\begin{thm}
\label{conj:multilinear}Suppose $Q$ is a quadratic form on $\mathbb{R}^{d}$
such that $Q$ is non-degenerate, indefinite with rational coefficients.
For $a\in\mathbb{Q}$ define $X_{\mathbb{Z}}=\left\{ x\in\mathbb{\mathbb{Z}}^{d}:Q\left(x\right)=a\right\} $,
suppose that $\left|X_{\mathbb{Z}}\right|=\infty$. Let $M=\left(L_{1},\ldots,L_{s}\right):\mathbb{R}^{d}\rightarrow\mathbb{R}^{s}$
be a linear map such that:
\begin{enumerate}
\item \label{enu:condition 1}The following inequalities hold, $d>2s$ and
$\textrm{rank}\left(Q|_{\ker\left(M\right)}\right)>2$. 
\item \label{enu:condition 2}The quadratic form $Q|_{\ker\left(M\right)}$
is indefinite.
\item \label{enu:condition 3}For all $\alpha\in\mathbb{R}^{s}\setminus\left\{ 0\right\} $,
$\alpha_{1}L_{1}+\dots+\alpha_{s}L_{s}$ is non rational. 
\end{enumerate}
Then $\overline{M\left(X_{\mathbb{Z}}\right)}=\mathbb{R}^{s}$. 
\end{thm}
The key feature of Theorem \ref{conj:multilinear}, that is exploited
in its proof, is that $X_{\mathbb{R}}$ has a large group of symmetries.
Moreover, there is a large subgroup, $H$, of this group that stabilises
$M$ and is generated by one parameter unipotent subgroups. This means
that the problem can be studied from a dynamical systems point of
view. This is done via the following Theorem of M. Ratner found in
\cite{MR1262705}. 
\begin{thm}[Ratner's Theorem%
\footnote{Theorem \ref{thm:(Ranghunthans-togological-conjec} is also known
as Ratner's orbit closure Theorem or Ranghunathan's topological conjecture
after it was conjectured by him in the 70's. It can be seen as a vast
generalisation of a Theorem of Hedlund concerning horocyclic flows
in $SL_{2}\left(\mathbb{R}\right)/SL_{2}\left(\mathbb{Z}\right)$.
It should also be noted that a special case of Theorem \ref{thm:(Ranghunthans-togological-conjec}
was proved by S.G. Dani and G. Margulis in \cite{MR1032925} after
partial results in this direction were obtained for use in the proofs
of the Oppenheim conjecture and its subsequent refinements.%
}]
\label{thm:(Ranghunthans-togological-conjec}Let $G$ be a connected
Lie group and $H$ a subgroup of $G$ generated by one parameter unipotent
subgroups. Then, given a lattice $\Gamma$ of $G$ and any $x\in G/\Gamma$,
the closure of the orbit $Hx$ is equal to the orbit of a closed connected
subgroup $F$, such that $H\leq F\leq G$.
\end{thm}
In order to make use of Ratner's Theorem in this context, we study
the action of $H$ on a suitable homogeneous space, $G/\Gamma$. Then
one must show that under the rationality constraints imposed by condition
\ref{enu:condition 3}, the subgroup $F$ is sufficiently large. This
is the strategy that enables us to prove Theorem \ref{conj:multilinear}. 
\begin{rem}
There is no reason to expect that the inequality $2s<d$ from condition
\ref{enu:condition 1} is necessary however it is needed in the proof
of Lemma \ref{lem:no invariant subspaces}. The inequality $\textrm{rank}\left(Q|_{\ker\left(M\right)}\right)>2$
is analogous to the condition that $d>2$ in the Oppenheim conjecture
and is probably necessary, although no counterexample has been found. 
\end{rem}

\begin{rem}
Condition \ref{enu:condition 2} is possibly stronger than is strictly
necessary, however it is a natural condition and comparable with conditions
imposed in \cite{MR2067128}. It implies the necessary condition that
the set $X_{\mathbb{R}}\cap\left\{ x\in\mathbb{R}^{d}:M\left(x\right)=b\right\} $,
for some $b\in\mathbb{R}^{s}$ is non compact. To see that this condition
is necessary, suppose $X_{\mathbb{R}}\cap\left\{ x\in\mathbb{R}^{d}:M\left(x\right)=b\right\} $
is compact. Hence $X_{\mathbb{R}}\cap\left\{ x\in\mathbb{R}^{d}:\left|M\left(x\right)-b\right|\leq\epsilon\right\} $
is also compact and therefore contains only finitely many integer
points. Hence if $b\notin M\left(\mathbb{Z}^{d}\right)$, we can make
$\epsilon$ small enough so that $X_{\mathbb{R}}\cap\left\{ x\in\mathbb{R}^{d}:\left|M\left(x\right)-b\right|\leq\epsilon\right\} $
contains no integer points, but then there exists an open set $B_{\epsilon}\left(b\right)\subset\mathbb{R}^{s}$
such that there is no $x\in X_{\mathbb{Z}}$ with $M\left(x\right)\in B_{\epsilon}\left(b\right)$.
\end{rem}

\begin{rem}
Condition \ref{enu:condition 3} is necessary since otherwise $M\left(\mathbb{Z}^{d}\right)$
would not even be dense in $\mathbb{R}^{s}$. 
\end{rem}

\section{Set up}

\subsection{A canonical form for the system.}

Given a general pair $\left(\mathcal{Q},\mathcal{M}\right)$ consisting
of a non degenerate quadratic form and a linear map on $\mathbb{R}^{d}$,
it is possible to use linear transformations to transform $\left(\mathcal{Q},\mathcal{M}\right)$
into something more manageable. For two pairs $\left(Q_{1},M_{1}\right)$
and $\left(Q_{2},M_{2}\right)$ we say $\left(Q_{1},M_{1}\right)\sim\left(Q_{2},M_{2}\right)$
if and only if there exist $g_{d}\in GL_{d}\left(\mathbb{R}\right)$
and $g_{s}\in GL_{s}\left(\mathbb{R}\right)$ such that $\left(Q_{1}\left(x\right),M_{1}\left(x\right)\right)=\left(Q_{2}\left(g_{d}x\right),g_{s}M_{2}\left(g_{d}x\right)\right)$
for all $x\in\mathbb{R}^{d}$. 

The following result, adapted from \cite{MR2067128}, is reproduced
below since it will be used to establish a more general form. 
\begin{lem}
\label{lem:prop 2 canoical 1 d}Every pair $\left(Q,L\right)$, where
$Q$ is a non degenerate quadratic form on $\mathbb{R}^{d}$ with
signature $\left(p,q\right)$, and $L$ is a non zero linear form
on $\mathbb{R}^{d}$, is equivalent to one and only one of the following
pairs:
\begin{enumerate}
\item If $\textrm{rank}\left(Q|_{\ker\left(L\right)}\right)=d-1$, then
either

\begin{enumerate}
\item $\left(Q,L\right)\sim\left(\sum_{i=1}^{p}x_{i}^{2}-\sum_{i=p+1}^{d}x_{i}^{2},x_{1}\right)$
\item $\left(Q,L\right)\sim\left(\sum_{i=1}^{p}x_{i}^{2}-\sum_{i=p+1}^{d}x_{i}^{2},x_{d}\right)$.
\end{enumerate}
\item If $\textrm{rank}\left(Q|_{\ker\left(L\right)}\right)=d-2$, then
$\left(Q,L\right)\sim\left(2x_{1}x_{d}+\sum_{i=2}^{p}x_{i}^{2}-\sum_{i=p+1}^{d-1}x_{i}^{2},x_{1}\right)$. 
\end{enumerate}
\end{lem}
\begin{proof}
By Sylvester's Law we can always transform $Q\sim\sum_{i=1}^{p}x_{i}^{2}-\sum_{i=p+1}^{d}x_{i}^{2}$.
Next by applying an element of $SO\left(p,q\right)$ to the system
it is possible to ensure that the coefficient of $x_{1}$ in $L(x)$
is non zero. Now use the transformation $x_{1}\rightarrow x_{1}+L'\left(x\right)$
for $L'$ a linear form in the remaining variables, to get that 
\[
\left(Q,L\right)\sim\left(x_{1}^{2}+x_{1}L''\left(x\right)+Q'\left(x\right),x_{1}\right)
\]
for $Q'$ and $L''$ a quadratic form and linear form respectively,
in the variables not including $x_{1}$. Note that $Q'$ has signature
$\left(p-1,q\right)$ or $\left(p,q-1\right)$ if $\textrm{rank}\left(Q|_{\ker\left(L\right)}\right)=d-1$
and signature $\left(p-1,q-1\right)$ if $\textrm{rank}\left(Q|_{\ker\left(L\right)}\right)=d-2$.
Suppose we are in the first case, apply a transformation in the variables
not including $x_{1}$, to get that 
\[
\left(Q,L\right)\sim\left(x_{1}^{2}+2\sum_{i=2}^{d}\alpha_{i}x_{i}x_{1}+\sum_{i=2}^{\hat{p}}x_{i}^{2}-\sum_{i=\hat{p}+1}^{d}x_{i}^{2},x_{1}\right),
\]
where $\hat{p}=p$ or $\hat{p}=p+1$. Next, we use transformations
of the form 
\[
x_{k}\rightarrow\begin{cases}
x_{k}-\alpha_{k}x_{1} & \textrm{ for }2\leq k\leq\hat{p}\\
x_{k}+\alpha_{k}x_{1} & \textrm{ for }\hat{p}+1\leq k\leq d
\end{cases}
\]
 to get 
\[
\left(Q,L\right)\sim\left(x_{1}^{2}\left(1-\sum_{i=2}^{\hat{p}}\alpha_{i}^{2}+\sum_{i=\hat{p}+1}^{d}\alpha_{i}^{2}\right)+\sum_{i=2}^{\hat{p}}x_{i}^{2}-\sum_{i=\hat{p}+1}^{d}x_{i}^{2},x_{1}\right).
\]
If the coefficient of $x_{1}^{2}$ is positive, then $\hat{p}=p$
and we see that we are in case 1a of the Lemma, similarly if the coefficient
of $x_{1}^{2}$ is negative, then $\hat{p}=p+1$ and we see that,
after relabelling $x_{d}\rightarrow x_{1}$ and $x_{1}\rightarrow x_{d}$,
we are in case 1b of the Lemma.

Suppose that $\textrm{rank}\left(Q|_{\ker\left(L\right)}\right)=d-2$,
apply a transformation in the variables not including $x_{1}$ to
get that 
\[
\left(Q,L\right)\sim\left(x_{1}^{2}+2\sum_{i=2}^{d}\alpha_{i}x_{i}x_{1}+\sum_{i=2}^{p}x_{i}^{2}-\sum_{i=p+1}^{d-1}x_{i}^{2},x_{1}\right).
\]
Next, for $2\leq k\leq d-1$ we use transformations of the form $x_{k}\rightarrow x_{k}\pm\alpha_{k}x_{1}$
to make $\alpha_{i}$ zero for all $i\neq d$. Note that $\alpha_{d}\neq0$,
otherwise $Q$ would be degenerate, so to finish off we use the transformation
$x_{d}\rightarrow\frac{1}{2\alpha_{d}}\left(2x_{d}-x_{1}\right)$
and we see that we are in the second case of the Lemma. 
\end{proof}
We can now prove the main Lemma of this section. 
\begin{lem}
\label{canonical}For any pair $\left(Q,M\right)$, where $Q$ is
a non-degenerate quadratic form on $\mathbb{R}^{d}$ with signature
$(p,q)$, and $M:\mathbb{R}^{d}\rightarrow\mathbb{R}^{s}$ is a linear
map of rank $s$, if $Q|_{\ker\left(M\right)}$ is indefinite then
$\left(Q,M\right)\sim\left(Q_{0},M_{0}\right)$, where
\begin{alignat*}{1}
Q_{0}\left(x\right) & =Q_{m+1,\ldots,s}\left(x\right)+2\sum_{i=1}^{m}x_{i}x_{s+r+n+i}+\sum_{i=s+1}^{s+r}x_{i}^{2}-\sum_{i=s+r+1}^{s+r+n}x_{i}^{2}\\
M_{0}\left(x\right) & =\left(x_{1},\ldots,x_{s}\right),
\end{alignat*}
$m=d-s-\textrm{rank}\left(Q|_{\ker\left(M\right)}\right)$ and $Q_{m+1,\ldots,s}\left(x\right)$
is a non degenerate quadratic form in variables $x_{m+1},\ldots,x_{s}$
with signature $\left(p',q'\right)$, such that the following relations
hold, $r=p-m-p'\geq1$ and $n=q-m-q'\geq1$.\end{lem}
\begin{proof}
We will show that the pair $\left(Q,M\right)$ is equivalent to the
pair
\begin{alignat*}{1}
Q_{0}'\left(x\right) & =Q_{m+1,\ldots,s}\left(x\right)+2\sum_{i=1}^{m}x_{i}x_{d-i+1}+\sum_{i=s+1}^{s+r}x_{i}^{2}-\sum_{i=s+r+1}^{s+r+n}x_{i}^{2}\\
M_{0}'\left(x\right) & =\left(x_{1},\ldots,x_{s}\right)
\end{alignat*}
which is readily seen to be equivalent to the pair $\left(Q_{0},M_{0}\right)$
after relabelling $x_{d-i-1}\rightarrow x_{s+r+n+i}$ for $1\leq i\leq m$.
We proceed by induction on $s$. For $s=1$ we know from Lemma \ref{lem:prop 2 canoical 1 d}
that the conclusion of the Lemma holds, so suppose the Lemma holds
for $s\leq k-1$. Let $s=k$, and suppose that $M=\left(L_{1},\ldots,L{}_{k}\right)$
for $L_{1},\ldots,L{}_{k}$ non zero linear forms on $\mathbb{R}^{d}$.

If $\textrm{rank}\left(Q|_{\ker\left(L_{1}\right)}\right)=d-1$, using
Lemma \ref{lem:prop 2 canoical 1 d} it is clear that we can transform
our system into 
\[
\left(\begin{array}{c}
Q\\
M
\end{array}\right)\sim\left(\begin{array}{c}
\sum_{i=1}^{p}x_{i}^{2}-\sum_{i=p+1}^{d}x_{i}^{2}\\
\left(x_{l},L'_{2},\ldots,L'_{k}\right)
\end{array}\right),
\]
where $L'_{2},\ldots,L'_{k}$ are linear forms on $\mathbb{R}^{d}$
and $l\in\left\{ 1,d\right\} $. Next we can eliminate the coefficient
of $x_{l}$ in $L'_{2},\ldots,L'_{k}$ by subtracting some multiple
of $x_{l}$. Then, relabel $x_{l}\rightarrow x_{1}$ and $x_{1}\rightarrow x_{l}$
and apply the inductive hypothesis to see that the conclusion of the
Lemma holds. 

If $\textrm{rank}\left(Q|_{\ker\left(L_{1}\right)}\right)=d-2$, using
Lemma \ref{lem:prop 2 canoical 1 d} we get 
\[
\left(\begin{array}{c}
Q\\
M
\end{array}\right)\sim\left(\begin{array}{c}
2x_{1}x_{d}+\sum_{i=2}^{p}x_{i}^{2}-\sum_{i=p+1}^{d-1}x_{i}^{2}\\
\left(x_{1},L'_{2},\ldots,L'_{k}\right)
\end{array}\right).
\]
Again we can eliminate the coefficient of $x_{1}$ in $L'_{2},\ldots,L'_{k}$
by subtracting some multiple of $x_{1}$. Suppose that the coefficient
of $x_{d}$ is zero for each $L'_{2},\ldots,L'_{k}$, in this case
we are in position to apply the inductive hypothesis and get to the
conclusion of the lemma. Suppose the coefficient of $x_{d}$ in $L'_{i}$
is non zero for some $2\leq i\leq k$, without loss of generality
suppose that $i=2$, in particular suppose $L'_{2}\left(x\right)=L''\left(x\right)+\alpha_{d}x_{d}$
for some linear form $L''$ in variables $x_{2},\ldots,x_{d-1}$.
Use a transformation of the form $x_{d}\rightarrow\frac{1}{\alpha_{d}}\left(x_{d}-L''\left(x\right)\right)$
to get that 
\[
\left(\begin{array}{c}
Q\\
M
\end{array}\right)\sim\left(\begin{array}{c}
\frac{2}{\alpha_{d}}x_{1}x_{d}+\sum_{i=2}^{d-1}x_{i}\beta_{i}x_{1}+\sum_{i=2}^{p}x_{i}^{2}-\sum_{i=p+1}^{d-1}x_{i}^{2}\\
\left(x_{1},x_{d},L'_{3},\ldots,L'_{k}\right)
\end{array}\right)
\]
Next, for $2\leq k\leq d-1$ we use transformations of the form $x_{k}\rightarrow x_{k}\pm\beta_{k}x_{1}$
to make $\beta_{i}$ zero for all $i$. After we have done this we
end up with
\[
\left(\begin{array}{c}
Q\\
M
\end{array}\right)\sim\left(\begin{array}{c}
Q_{1,d}\left(x\right)+\sum_{i=2}^{p}x_{i}^{2}-\sum_{i=p+1}^{d-1}x_{i}^{2}\\
\left(x_{1},x_{d},L''_{3},\ldots,L''_{k}\right)
\end{array}\right),
\]
where $Q_{1,d}$ is a quadratic form in the variables $x_{1}$ and
$x_{d}$ and $L''_{3},\dots,L''_{k}$ are linear forms in variables
$x_{1},\ldots,x_{d}$. If we relabel $x_{2}\rightarrow x_{d}$ and
$x_{d}\rightarrow x_{2}$ and eliminate the $x_{2}$ and $x_{1}$
co-ordinates from $L''_{3},\ldots,L''_{k}$ we can apply the inductive
hypothesis and get the desired conclusion. The assertion that $Q_{m+1,\ldots,s}\left(x\right)$
is a non degenerate quadratic form in variables $x_{m+1},\ldots,x_{s}$
follows from the fact that $Q$ is non degenerate. We see that $Q_{m+1,\ldots,s}\left(x\right)$
has signature $\left(p',q'\right)$ where $r=p-m-p'$ and $n=q-m-q'$
because the signature of $2\sum_{i=1}^{m}x_{i}x_{d-i+1}+\sum_{i=s+1}^{s+r}x_{i}^{2}-\sum_{i=s+r+1}^{s+r+n}x_{i}^{2}$
is $\left(r+m,n+m\right)$. Finally, the assumption that $Q|_{\ker\left(M\right)}$
is indefinite means that $r\geq1$ and $n\geq1$.
\end{proof}

\subsection{Construction of a dynamical system}

In order to make use of Theorem \ref{thm:(Ranghunthans-togological-conjec}
we need to construct a dynamical system. For any pair $\left(\mathcal{Q},\mathcal{M}\right)$
consisting of a non degenerate quadratic form and a linear map on
$\mathbb{R}^{d}$. Define $G_{\mathcal{Q}}$ to be the connected component
containing the identity of $\left\{ g\in SL_{d}\left(\mathbb{R}\right):\mathcal{Q}\left(gx\right)=\mathcal{Q}\left(x\right)\right\} $.
Let $\Gamma_{\mathcal{Q}}=G_{\mathcal{Q}}\cap SL_{d}\left(\mathbb{Z}\right)$
and $H_{\mathcal{Q},\mathcal{M}}=\left\{ g\in G_{\mathcal{Q}}:\mathcal{M}\left(gx\right)=\mathcal{M}\left(x\right)\right\} $. 

Once and for all, fix a pair $\left(Q,M\right)$ consisting of a quadratic
form and a linear map satisfying the conditions of Theorem \ref{conj:multilinear}.
Suppose that $Q$ has signature $\left(p,q\right)$. It is a standard
fact that $G_{Q}\cong SO\left(p,q\right)^{o}$ is a connected Lie
group. Since $Q$ is a rational form, $G_{Q}$ is defined over the
rationals. Because $G_{Q}$ is semisimple and therefore does not admit
any non trivial rational characters, this means that $\Gamma_{Q}$
is a lattice in $G_{Q}$ (cf. \cite{MR1278263}, Theorem 4.13). Since
a priori, $H_{Q,M}$ may not be generated by one parameter unipotent
subgroups, our first aim is to define $H_{Q,M}^{*}\leq H_{Q,M}$ such
that $H_{Q,M}^{*}$ is generated by unipotent subgroups, we will then
consider the dynamical system that arises from $H_{Q,M}^{*}$ acting
on $G_{Q}/\Gamma_{Q}$. Note that condition \ref{enu:condition 2}
of Theorem \ref{conj:multilinear} implies that $H_{Q,M}$ will be
non compact, and so there is hope that such an $H_{Q,M}^{*}$ exists,
in section \ref{sub:Definition-of H} an explicit description of $H_{Q,M}^{*}$
is given. 

Let the pair $\left(Q_{0},M_{0}\right)$ be as defined in Lemma \ref{canonical}.
Let $g_{d}\in GL_{d}\left(\mathbb{R}\right)$ and $g_{s}\in GL_{ds}\left(\mathbb{R}\right)$
be such that $\left(Q\left(x\right),M\left(x\right)\right)=\left(Q_{0}\left(g_{d}x\right),g_{s}M_{0}\left(g_{d}x\right)\right)$
for all $x\in\mathbb{R}^{d}$. We will use the shorthand $G_{Q_{0}}=G_{0}$,
$\Gamma_{Q_{0}}=\Gamma_{0}$ and $H_{Q_{0},M_{0}}=$$H_{0}$.

\subsection{Definition of $H_{Q,M}^{*}$\label{sub:Definition-of H}}

For non negative integers $z_{1}$ and $z_{2}$ we will use the notation
$I_{z_{1}}$ to denote the $z_{1}\times z_{1}$ identity matrix and
$I_{z_{1},z_{2}}$ to denote $\left(\begin{smallmatrix}I_{z_{1}}\\
 & -I_{z_{2}}
\end{smallmatrix}\right)$ . Also, $\mbox{Mat}_{z_{1}}\left(\mathbb{R}\right)$ denotes square,
$z_{1}\times z_{1}$, matrices with entries in $\mathbb{R}$ and $\mbox{Mat}_{z_{1},z_{2}}\left(\mathbb{R}\right)$
denotes matrices with $z_{1}$ rows, $z_{2}$ columns and entries
in $\mathbb{R}$.  For any matrix, $m$ the notation $m^{T}$ is used
to denote the transpose of $m$. Let $O\left(z_{1},z_{2}\right)=\left\{ g\in GL_{z_{1}+z_{2}}\left(\mathbb{R}\right):g^{T}I_{z_{1},z_{2}}g=I_{z_{1},z_{2}}\right\} $
and $SO\left(z_{1},z_{2}\right)=O\left(z_{1},z_{2}\right)\cap SL_{z_{1}+z_{2}}\left(\mathbb{R}\right)$.
Let the parameters $p',q',r,n$ and $m$ be as defined in Lemma \ref{canonical}.
Let $i_{1}$,$i_{2}$ and $i_{3}$ be integers such that $0\leq i_{1}\leq p'$,
$0\leq i_{2}\leq q'$ and $-\min\left\{ p',q'\right\} \leq i_{3}\leq m$.
Let $B=\mbox{Mat}_{m-i_{3},r+n+i_{1}+i_{2}+2i_{3}}\left(\mathbb{R}\right)$
and for $t\in B$ let 
\[
B\left(t\right)=\left\{ m\in\mbox{Mat}_{m-i_{3}}\left(\mathbb{R}\right):m+m^{T}+tI_{r+i_{1}+i_{3},n+i_{2}+i_{3}}t^{T}=0\right\} .
\]
Now define 
\[
D_{i_{1},i_{2},i_{3}}=\left(\begin{array}{ccc}
I_{s-i_{1}-i_{2}-i_{3}} & 0 & 0\\
0 & SO\left(r+i_{1}+i_{3},n+i_{2}+i_{3}\right)^{o} & 0\\
0 & 0 & I_{m-i_{3}}
\end{array}\right).
\]
and 
\[
U_{i_{1},i_{2},i_{3}}=\left\{ \begin{array}{l}
\left(\begin{array}{cccc}
I_{m-i_{3}} & 0 & 0 & 0\\
0 & I_{s-m-i_{1}-i_{2}} & 0 & 0\\
-I_{r+i_{1}+i_{3},n+i_{2}+i_{3}}t^{T} & 0 & I_{r+n+i_{1}+i_{2}+2i_{3}} & 0\\
s & 0 & t & I_{m-i_{3}}
\end{array}\right)\end{array}:t\in B,s\in B\left(t\right)\right\} .
\]
For convenience we denote $D_{0,0,0}=D$ and $U_{0,0,0}=U$. It is
straightforward to verify that $U$ and $D$ are subgroups of $H_{0}$
and that $U$ is normalised by $D$. Define 
\[
H_{0}^{*}=UD.
\]
More generally $U_{i_{1},i_{2},i_{3}}$ is normalised by $D_{i_{1},i_{2},i_{3}}$
and in Section 4 we will use $\left(UD\right){}_{i_{1},i_{2},i_{3}}$
to denote the subgroup $U_{i_{1},i_{2},i_{3}}D_{i_{1},i_{2},i_{3}}$.
Conditions \ref{enu:condition 1} and \ref{enu:condition 2} of Theorem
\ref{conj:multilinear} imply that $D$ is generated by one parameter
unipotent subgroups. To see this note that $\textrm{rank}\left(Q|_{\ker\left(M\right)}\right)>2$
implies that $r+n\geq3$. Moreover, as noted in Lemma \ref{canonical},
the fact that $Q|_{\ker\left(M\right)}$ is indefinite implies that
$r\geq1$ and $n\geq1$. Therefore, since $U$ is a unipotent subgroup,
$H_{0}^{*}$ defined in this way is connected and generated by one
parameter unipotent subgroups. Define 
\[
H_{Q,M}^{*}=g_{d}H_{0}^{*}g_{d}^{-1}.
\]
Now it is clear that $H_{Q,M}^{*}\leq H_{Q,M}$ and is generated by
one parameter unipotent subgroups as required.

\section{Lemmas concerning invariant subspaces.}

Let $\mathcal{L}$ be the space of $d$ dimensional linear forms defined
over $\mathbb{\mathbb{\mathbb{C}}}$. Let $e_{1},\dots,e_{d}$ be
the standard basis of $\mathbb{C}^{d}$ and $x_{1},\dots,x_{d}$ be
the corresponding basis of $\mathcal{L}$. Let $M=\left(L_{1},\ldots,L_{s}\right)$
for $L_{i}\in\mathcal{L}$. For any group $\mathcal{G}$ and any set
$\mathcal{S}$ with a well defined action, $\mathcal{G}\times\mathcal{S}\rightarrow\mathcal{S}$,
let $\mathcal{S}^{\mathcal{G}}=\left\{ s\in\mathcal{S}:g.s=s\textrm{ for all }g\in\mathcal{G}\right\} $.
The subspace $\mathcal{L}^{\mathcal{G}}$ will be referred to as the
`fixed vectors of $\mathcal{G}$'. 

We will say that a subspace $\mathcal{V}$ of $\mathcal{L}$ is defined
over $\mathbb{Q}$ if there exists a basis for $\mathcal{V}$ where
each basis vector is in $\mathbb{Q}^{d}$. First we classify the space
of vectors fixed by $H_{Q,M}^{*}$. For any group $\mathcal{G}$ let
the action $\mathcal{G}\times\mathcal{L}\rightarrow\mathcal{L}$ be
defined by $\left(g,L\left(x\right)\right)\rightarrow L\left(gx\right)=g^{T}l.x$
where $l\in\mathbb{C}^{d}$ such that $L\left(x\right)=l.x$. From
now on, this action will just be denoted by $gL$. 
\begin{lem}
\label{lem:Forms fixed by H star} $\mathcal{L}^{H_{Q,M}^{*}}=\left\langle L_{1},\ldots,L_{s}\right\rangle $. \end{lem}
\begin{proof}
It is clear that $\left\langle L_{1},\ldots,L_{s}\right\rangle \subseteq\mathcal{L}^{H_{Q,M}^{*}}$.
Suppose there exists $L\in\mathcal{L}^{H_{Q,M}^{*}}$ such that $L\notin\left\langle L_{1},\ldots,L_{s}\right\rangle $,
or equivalently, there exists $L\in\mathcal{L}^{H_{0}^{*}}$ such
that $L\notin\left\langle x_{1},\ldots,x_{s}\right\rangle $. Let
$l\in\mathbb{C}^{d}$ and write $L\left(x\right)=l.x$, since $D<H_{0}^{*}$
we may suppose that $l_{s+1},\ldots,l_{d-m}=0$ as $SO\left(r,n\right)^{o}$
has no fixed vectors. But we also have $U<H_{0}^{*}$ and for every
non zero $L\in\left\langle x_{d-m+1},\ldots,x_{d}\right\rangle $
there exists $u\in U$ such that $uL\notin\left\langle x_{d-m+1},\ldots,x_{d}\right\rangle $.
This implies that $l_{d-m+1},\ldots,l_{d}=0$, and thus $L\in\left\langle x_{1},\ldots,x_{s}\right\rangle $
and we have a contradiction. 
\end{proof}
Recall, condition \ref{enu:condition 3} of Theorem \ref{conj:multilinear}
is that for all $\alpha\in\mathbb{R}^{s}\setminus\left\{ 0\right\} $,
$\alpha_{1}L_{1}+\dots+\alpha_{s}L_{s}$ is non rational. From this
and Lemma \ref{lem:Forms fixed by H star} we can deduce the following.
\begin{cor}
\label{cor:rationality contrad}There exists no non trivial subspaces
defined over $\mathbb{Q}$ and contained in $\mathcal{L}^{H_{Q,M}^{*}}$.\end{cor}
\begin{proof}
Suppose there exists $\mathcal{U}\subseteq\mathcal{L}^{H_{Q,M}^{*}}$
such that $\mathcal{U}$ is defined over $\mathbb{Q}$. Lemma \ref{lem:Forms fixed by H star}
implies that $\mathcal{U}\subseteq\left\langle L_{1},\ldots,L_{s}\right\rangle $.
The fact that $\mathcal{U}$ is defined over $\mathbb{Q}$ means there
exists $u_{1},\ldots,u_{\dim\left(\mathcal{U}\right)}\in\mathbb{Q}^{d}$
such that $\mathcal{U}=\left\langle u_{1},\ldots,u_{\dim\left(\mathcal{U}\right)}\right\rangle $.
Therefore we can write each $u_{i}$ as a linear combination of the
$L_{i}$'s. Since the $u_{i}$ are rational this contradicts condition
\ref{enu:condition 3} of Theorem \ref{conj:multilinear} as required.
\end{proof}
Let $\mathcal{L}_{m}=\left\langle x_{s+1},\ldots,x_{d-m}\right\rangle $
where $m=d-s-\textrm{rank}\left(Q|_{\ker\left(M\right)}\right)$ as
defined in Lemma \ref{canonical}, in particular since $\textrm{rank}\left(Q|_{\ker\left(M\right)}\right)>2$,
$\dim\left(\mathcal{L}_{m}\right)=d-m-s>2$. The next Lemma classifies
two distinct possibilities for any $D$ invariant subspace. 
\begin{lem}
\label{lem:D invarient subspaces}If $\mathcal{V}\subseteq\mathcal{L}$
is a $D$ invariant subspace, then either
\begin{enumerate}
\item $\mathcal{V}\subseteq\mathcal{L}^{D}$
\item $\mathcal{V}=\mathcal{L}_{m}\oplus\mathcal{U}$ where $\mathcal{U}\subseteq\mathcal{L}^{D}$. 
\end{enumerate}
\end{lem}
\begin{proof}
First note that if $\mathcal{V}\subseteq\mathcal{L}^{D}$, then it
is clear that $\mathcal{V}$ will be $D$ invariant. Suppose that
$\mathcal{V}\subseteq\mathcal{L}$ is a $D$ invariant linear subspace
such that $\mathcal{V}\nsubseteq\mathcal{L}^{D}$, since $\mathcal{L}=\mathcal{L}^{D}\oplus\mathcal{L}_{m}$
there exists $v\in\mathcal{V}$ such that $v=v_{1}+v_{2}$ for $v_{1}\in\mathcal{L}^{D}$
and $v_{2}\in\mathcal{L}_{m}$, with $v_{2}\neq0$. Now for any $d\in D$
we have $dv-v=v_{1}+dv_{2}-\left(v_{1}+v_{2}\right)=dv_{2}-v_{2}\in\mathcal{V}$,
since $\mathcal{L}^{D}$ consists of forms fixed by $D$. We can choose
$d\in D$ so that $w=dv_{2}-v_{2}\neq0$, but $w\in\mathcal{V}\cap\mathcal{L}_{m}$,
this means that $\left\langle Dw\right\rangle $ is a $D$ invariant
subspace such that $\left\langle Dw\right\rangle \subseteq\mathcal{L}_{m}$,
but this implies $\left\langle Dw\right\rangle =\mathcal{L}_{m}$
because $D$ acts irreducibly on $\mathcal{L}_{m}$. Since $\left\langle Dw\right\rangle =\mathcal{L}_{m}$,
we have that $\mathcal{L}_{m}\subseteq\mathcal{V}$ and as $v=v_{1}+v_{2}\in\mathcal{V}$
for $v_{1}\in\mathcal{L}^{D}$ and $v_{2}\in\mathcal{L}_{m}$ we have
$v_{1}=v-v_{2}\in\mathcal{V}$ and so we see that $\mathcal{V}=\mathcal{L}_{m}\oplus\left(\mathcal{V}\cap\mathcal{L}^{D}\right)$,
which implies that we are in the second case. 
\end{proof}
The next Lemma extends Lemma \ref{lem:D invarient subspaces} to classify
two distinct possibilities for any $H_{0}^{*}$ invariant subspace
of $\mathcal{L}$. For convenience, let $\mathcal{L}_{0}=\left\langle x_{s+1},\ldots,x_{d}\right\rangle $
and $J_{m}=\left(\begin{smallmatrix}0 & 0 & I_{m}\\
0 & I_{d-2m} & 0\\
I_{m} & 0 & 0
\end{smallmatrix}\right)$ where $J_{0}=I_{d}$. 
\begin{lem}
\label{lem:general h hat star inv subspaces}If $\mathcal{V}\subseteq\mathcal{L}$
is a $H_{0}^{*}$ invariant subspace, then either
\begin{enumerate}
\item $\mathcal{V}\subseteq\mathcal{L}^{H_{0}^{*}}$
\item $\mathcal{V}=J_{m}\mathcal{L}_{0}\oplus\mathcal{U}$ where $\mathcal{U}\subseteq J_{m}\mathcal{L}^{H_{0}^{*}}$
. 
\end{enumerate}
\end{lem}
\begin{proof}
First note that if $\mathcal{V}\subseteq\mathcal{L}^{H_{0}^{*}}$,
then it is clear that $\mathcal{V}$ will be $H_{0}^{*}$ invariant.
Suppose that $\mathcal{V}\subseteq\mathcal{L}$ is an $H_{0}^{*}$
invariant linear subspace such that $\mathcal{V}\nsubseteq\mathcal{L}^{H_{0}^{*}}$,
since $D\leq H_{0}^{*}$ it is clear $\mathcal{V}$ will be $D$ invariant,
thus by Lemma \ref{lem:D invarient subspaces} either 
\begin{enumerate}
\item $\mathcal{V}\subseteq\mathcal{L}^{D}$ 
\item $\mathcal{V}=\mathcal{L}_{m}\oplus\mathcal{U}$ where $\mathcal{U}\subseteq\mathcal{L}^{D}$.
\end{enumerate}
If we are in case (1), then since $\mathcal{L}^{D}=\mathcal{L}^{H_{0}^{*}}\oplus\left\langle x_{d-m+1},\ldots,x_{d}\right\rangle $
we can suppose there exists $v\in\mathcal{V}$ such that $v=v_{1}+v_{2}$
for $v_{1}\in\mathcal{L}^{H_{0}^{*}}$ and $v_{2}\in\left\langle x_{d-m+1},\ldots,x_{d}\right\rangle $
such that $v_{2}\neq0$. Then there exists $u\in U<H_{0}^{*}$ such
that $uv_{2}\notin\mathcal{L}^{D}$ and hence $uv_{2}\notin\mathcal{V}$,
but this is a contradiction since $\mathcal{V}$ is supposed to be
$H_{0}^{*}$ invariant. 

If we are in case (2), from the definitions of $U$, $\mathcal{L}_{m}$
and $J_{m}$ we see that $\left\langle U\mathcal{L}_{m}\right\rangle =\left\langle J_{m}\mathcal{L}_{0}\right\rangle $,
then since $U<H_{0}^{*}$ , the fact that $\mathcal{L}_{m}\subseteq\mathcal{V}$
and $\mathcal{V}$ is $H_{0}^{*}$ invariant implies that $\mathcal{V}=J_{m}\mathcal{L}_{0}+\mathcal{U}$
for some $\mathcal{U}\subseteq\mathcal{L}$. Moreover, since $\mathcal{L}=J_{m}\mathcal{L}^{H_{0}^{*}}\oplus J_{m}\mathcal{L}_{0}$
we see that $\mathcal{V}=J_{m}\mathcal{L}_{0}\oplus\left(\mathcal{V}\cap J_{m}\mathcal{L}^{H_{0}^{*}}\right)$
which implies that we are in the second case of the Lemma. 
\end{proof}
For fields $\mathbb{F}_{1}$ and $\mathbb{F}_{2}$ such that $\mathbb{F}_{1}\subseteq\mathbb{F}_{2}$,
the notation $\textrm{Aut}\left(\mathbb{F}_{2}/\mathbb{F}_{1}\right)$
is used to stand for the group of automorphisms of $\mathbb{F}_{2}$
that fix $\mathbb{F}_{1}$. Let $\overline{\mathbb{Q}}$ denote the
algebraic closure of $\mathbb{Q}$. It is a standard fact that $\overline{\mathbb{Q}}^{\textrm{Aut}\left(\mathbb{\overline{\mathbb{Q}}}/\mathbb{Q}\right)}=\mathbb{Q}$
(cf. \cite{MR1102012}, page 30). For the present situation the corresponding
fact for $\textrm{Aut}\left(\mathbb{C}/\mathbb{Q}\right)$ is needed.
This is probably well known, but no good references were found, so
proofs are included. It should be noted that the argument of Lemma
\ref{lem:rat basis} is adapted from that used on page 30 of \cite{MR1102012}. 
\begin{lem}
\label{lem fixed vectors of big galios group}$\mathbb{C}^{\textrm{Aut}\left(\mathbb{C}/\mathbb{Q}\right)}=\mathbb{Q}$.\end{lem}
\begin{proof}
It is clear that $\mathbb{Q}\subseteq\mathbb{C}^{\textrm{Aut}\left(\mathbb{C}/\mathbb{Q}\right)}$.
By Theorem 7 of \cite{1966} any element of $\textrm{Aut}\left(\overline{\mathbb{Q}}/\mathbb{Q}\right)$
can be extended to an element of $\textrm{Aut}\left(\mathbb{C}/\mathbb{Q}\right)$
and hence $\mathbb{\overline{Q}}\cap\mathbb{C}^{\textrm{Aut}\left(\mathbb{C}/\mathbb{Q}\right)}=\mathbb{Q}$.
With this in mind it suffices to show that for any $x\in\mathbb{C}\setminus\overline{\mathbb{Q}}$
there exists \foreignlanguage{english}{$\sigma\in\textrm{Aut}\left(\mathbb{C}/\mathbb{Q}\right)$}
such that \foreignlanguage{english}{$\sigma\left(x\right)\neq x$}.
Let $x\in\mathbb{C}\setminus\overline{\mathbb{Q}}$, there exists
a transcendence basis $S$ for $\mathbb{C}/\mathbb{Q}$ such that
$x\in S$. There exists an automorphism, $\phi$ of $\mathbb{Q}\left(S\right)$
that acts by permuting elements of $S$. Since $\mathbb{Q}\left(S\right)$
is a subfield of $\mathbb{C}$, again by Theorem 7 of \cite{1966},
$\phi$ can be extended to an automorphism of $\mathbb{C}/\mathbb{Q}$
and we are done. 
\end{proof}
\textcolor{black}{For any }$\sigma\in\textrm{Aut}\left(\mathbb{C}/\mathbb{Q}\right)$
write $\sigma\left(\mathcal{V}\right)$\textcolor{black}{{} to mean
the vector space with the basis obtained by applying $\sigma$ to
all components of all basis vectors of $\mathcal{V}$. }
\begin{lem}
\label{lem:rat basis}Suppose $\mathcal{V}\subseteq\mathbb{C}^{d}$
is such that $\sigma\left(\mathcal{V}\right)=\mathcal{V}$ for all
$\sigma\in\textrm{Aut}\left(\mathbb{C}/\mathbb{Q}\right)$, then $\mathcal{V}$
is defined over $\mathbb{Q}$. \end{lem}
\begin{proof}
Suppose that $\mathcal{V}$ is not defined over $\mathbb{Q}$. In
particular this means that $\mathcal{V}$ does not contain any vectors
of the form $\lambda q$ where $\lambda\in\mathbb{C}$ and $q\in\mathbb{Q}^{d}$,
in particular $\mathcal{V}\neq0$. Since $\mathcal{V}\neq0$ we can
choose $v\in\mathcal{V}$ such that $v$ is a linear combination of
the least possible number of the $e_{j}$'s. After renumbering the
$e_{j}'s$ and multiplying $v$ by an element of $\mathbb{C}$ we
can suppose that $v=e_{1}+v_{2}e_{2}+\dots$, with $v_{2}\notin\mathbb{Q}$
since otherwise $v\in\mathbb{Q}^{d}$, contradicting the fact that
$\mathcal{V}$ is not defined over $\mathbb{Q}$. Then Lemma \ref{lem fixed vectors of big galios group}
implies that there exists $\sigma\in\textrm{Aut}\left(\mathbb{C}/\mathbb{Q}\right)$
such that $\sigma\left(v_{2}\right)\neq v_{2}$ and hence, because
$\sigma\left(\mathcal{V}\right)=\mathcal{V}$ for all $\sigma\in\textrm{Aut}\left(\mathbb{C}/\mathbb{Q}\right)$
we get that $v-\sigma\left(v\right)\in\mathcal{V}$ and $v-\sigma\left(v\right)\neq0$.
It is also clear that $v-\sigma\left(v\right)$ can be written as
linear combination of fewer $e_{i}$'s than $v$ contradicting our
choice of $v$. 
\end{proof}
During the proof of the following Lemma, the assumption that $d>2s$
becomes essential. 
\begin{lem}
\label{lem:no invariant subspaces}Let $F$ be a closed connected
subgroup such that $H_{Q,M}^{*}\leq F\leq SL_{d}\left(\mathbb{R}\right)$
and $\overline{F\cap SL_{d}\left(\mathbb{Q}\right)}=F$. If $\mathcal{V}\subseteq\mathcal{L}$
is a non trivial $F$ invariant subspace, then $\dim\left(\mathcal{V}\right)\geq d-s$. \end{lem}
\begin{proof}
Let $\mathcal{V}\subseteq\mathcal{L}$ be a non trivial $F$ invariant
subspace. Since $H_{Q,M}^{*}$ is a subgroup of $F$, $\mathcal{V}$
is also $H_{Q,M}^{*}$ invariant. This means that then $g_{d}^{-1}\mathcal{V}$
is a $H_{0}^{*}$ invariant subspace, hence Lemma \ref{lem:general h hat star inv subspaces}
implies that $g_{d}^{-1}\mathcal{V}\subseteq\mathcal{L}^{H_{0}^{*}}$
or $g_{d}^{-1}\mathcal{V}=J_{m}\mathcal{L}_{0}\oplus\mathcal{U}$
where $\mathcal{U}\subseteq J_{m}\mathcal{L}^{H_{0}^{*}}$. In the
second case $\dim\left(\mathcal{V}\right)\geq\dim\left(J_{m}\mathcal{L}_{0}\oplus\mathcal{U}\right)\geq\dim\left(\mathcal{L}_{0}\right)=d-s$,
which is the conclusion of the Lemma. Therefore, it is sufficient
to show that there are no non trivial $F$ invariant subspaces contained
in $g_{d}\mathcal{L}^{H_{0}^{*}}=\mathcal{L}^{H_{Q,M}^{*}}$. Suppose
for a contradiction there exists at least one non trivial $F$ invariant
subspace contained in $\mathcal{L}^{H_{Q,M}^{*}}$. Define $\mathcal{V}$
to be the unique maximal $F$ invariant subspace such that $\mathcal{V}\subseteq\mathcal{L}^{H_{Q,M}^{*}}$.
\textcolor{black}{For any }$\sigma\in\textrm{Aut}\left(\mathbb{C}/\mathbb{Q}\right)$
we have that $\sigma\left(\mathcal{V}\right)$ is $F\cap SL_{d}\left(\mathbb{Q}\right)$
invariant and by our assumption on $F$ we have $\overline{F\cap SL_{d}\left(\mathbb{Q}\right)}=F$,
and therefore $\sigma\left(\mathcal{V}\right)$ is $F$ invariant.
Additionally, this means that $\sigma\left(\mathcal{V}\right)$ is
$H_{Q,M}^{*}$ invariant or equivalently $g_{d}^{-1}\sigma\left(\mathcal{V}\right)$
is $H_{0}^{*}$ invariant, hence Lemma \ref{lem:general h hat star inv subspaces}
implies, either $g_{d}^{-1}\sigma\left(\mathcal{V}\right)\subseteq\mathcal{L}^{H_{0}^{*}}$,
or $g_{d}^{-1}\sigma\left(\mathcal{V}\right)=J_{m}\mathcal{L}_{0}\oplus\mathcal{U}$
where $\mathcal{U}\subseteq J_{m}\mathcal{L}^{H_{0}^{*}}$. But $\dim\left(g_{d}^{-1}\sigma\left(\mathcal{\mathcal{V}}\right)\right)\leq s<d-s\leq\dim\left(J_{m}\mathcal{L}_{0}\oplus\mathcal{U}\right)$,
this means that $g_{d}^{-1}\sigma\left(\mathcal{\mathcal{V}}\right)\subseteq\mathcal{L}^{H_{0}^{*}}$
or in other words $\sigma\left(\mathcal{\mathcal{V}}\right)\subseteq\mathcal{L}^{H_{Q,M}^{*}}$.
From the definition of $\mathcal{V}$, we see that $\sigma\left(\mathcal{V}\right)\subseteq\mathcal{V}$
and therefore by considering dimensions $\sigma\left(\mathcal{V}\right)=\mathcal{V}$.
Hence Lemma \ref{lem:rat basis} implies that $\mathcal{V}$ is defined
over $\mathbb{Q}$, contradicting Corollary \ref{cor:rationality contrad}. 
\end{proof}

\section{Proof of the main Theorem.}

To prove Theorem \ref{conj:multilinear} it suffices to prove the
following. 
\begin{prop}
\label{cor:hopefully teh eend}If $F$ is a closed connected subgroup
such that $H_{Q,M}^{*}\leq F\leq G_{Q}$ and $F$ has no non trivial
invariant subspaces of dimension less than $d-s$. Then $F=G_{Q}$. 
\end{prop}
Once this knowledge is available, the results of the previous section
and general facts about algebraic groups can be used together with
Ratner's Theorem to obtain the relation $\overline{H_{Q,M}^{*}x}=G_{Q}x=G_{Q}$
for $x=e\Gamma_{Q}$. It is straightforward to show that the latter
relation implies Theorem \ref{conj:multilinear}. The main body to
this section is devoted to a proof of Proposition \prettyref{cor:hopefully teh eend},
in the course of this proof some cumbersome notation is used, possibly
obscuring the underlying idea, therefore an outline of the proof is
presented as follows. 

The main step is to show that there is some conjugate of $F$ that
contains an enlarged copy of $H_{Q,M}^{*}$. Once this is proven,
the procedure can be repeated, and at each stage there is some conjugate
of $F$ that contains a larger subgroup of the same form as $H_{Q,M}^{*}$.
In order to prove the former claim, presented as Lemma \ref{lem:f=00003Dso(p,q) general}
below, the Lie algebra of $F$ is decomposed into subspaces defined
in terms of $4$ by $4$ block matrices. Then it is shown that, if
the intersection of these subspaces with the Lie algebra of $F$ is
trivial in certain cases, then $F$ will have invariant subspaces
of dimension less than $d-s$, contradicting the assumptions on $F$.
This means that the intersection of these subspaces with the Lie algebra
of $F$ is non trivial and by conjugating we can rearrange the subspaces,
contained in the intersection, in a form that shows $F$ must contain
an enlarged copy of $H_{Q,M}^{*}$. 

We now proceed with the actual proof. 
\begin{lem}
\label{lem:f=00003Dso(p,q) general}Let the parameters $p',q',r,n$
and $m$ be as defined in Lemma \ref{canonical} and let $i_{1}$,$i_{2}$
and $i_{3}$ be integers such that $0\leq i_{1}\leq p'$, $0\leq i_{2}\leq q'$
and $-\min\left\{ p',q'\right\} \leq i_{3}\leq m$ and $i_{1}+i_{2}+i_{3}<p'+q'+m$.
Let $Q_{i_{1},i_{2},i_{3}}$ be the quadratic form defined by the
matrix
\[
Q'{}_{i_{1},i_{2},i_{3}}=\left(\begin{array}{cccc}
0 & 0 & 0 & I_{m-i_{3}}\\
0 & I_{p'-i_{1},q'-i_{2}} & 0 & 0\\
0 & 0 & I_{r+i_{1}+i_{3},n+i_{2}+i_{3}} & 0\\
I_{m-i_{3}} & 0 & 0 & 0
\end{array}\right).
\]
Let $F$ be a closed connected subgroup such that $\left(UD\right){}_{i_{1},i_{2},i_{3}}\leq F\leq G_{Q_{i_{1},i_{2},i_{3}}}$
and $F$ has no non trivial invariant subspaces of dimension less
than $d-s$. Then there exists $\eta\in GL_{d}\left(\mathbb{R}\right)$
with the properties that either:
\begin{enumerate}
\item $\left(UD\right){}_{i_{1}+1,i_{2},i_{3}}\leq\eta F\eta^{-1}\leq G_{Q_{i_{1}+1,i_{2},i_{3}}}$,
\item $\left(UD\right){}_{i_{1},i_{2}+1,i_{3}}\leq\eta F\eta^{-1}\leq G_{Q_{i_{1},i_{2}+1,i_{3}}}$,
\item $\left(UD\right){}_{i_{1},i_{2},i_{3}+1}\leq\eta F\eta^{-1}\leq G_{Q_{i_{1},i_{2},i_{3}+1}}$,
\selectlanguage{english}%
\item \textup{$\left(UD\right){}_{i_{1}+1,i_{2}+1,i_{3}-1}\leq\eta F\eta^{-1}\leq G_{Q_{i_{1}+1,i_{2}+1,i_{3}-1}}.$}
\end{enumerate}
Moreover, if $i_{1}=p'$ then only cases 2 and 3 can occur, if $i_{2}=q'$
then only cases 1 and 3 occur, if $i_{3}=m$ then case 3 will not
occur and if $i_{3}=-\min\left\{ p',q'\right\} $ then case 4 will
not occur. \end{lem}
\begin{proof}
Everything that follows depends on the parameters $i_{1},i_{2}$ and
$i_{3}$, but in order to make the notation more digestible let $m-i_{3}=l$,
$r+i_{1}+i_{3}=\sigma_{1}$, $n+i_{2}+i_{3}=\sigma_{2}$, $p'-i_{1}=\tau_{1}$
, $q'-i_{2}=\tau_{2}$ , $\tau_{1}+\tau_{2}=\tau$ and $\sigma_{1}+\sigma_{2}=\sigma$.
Let 
\[
f=\left(\begin{array}{cccc}
f_{11} & f_{12} & f_{13} & f_{14}\\
f_{21} & f_{22} & f_{23} & f_{24}\\
f_{31} & f_{32} & f_{33} & f_{34}\\
f_{41} & f_{42} & f_{43} & f_{44}
\end{array}\right)
\]
where $f_{11},f_{14},f_{41},f_{44}\in\mbox{Mat}_{l}\left(\mathbb{R}\right)$,
$f_{21}^{T},f_{12},f_{24}^{T},f_{42}\in\mbox{Mat}_{l,\tau}\left(\mathbb{R}\right)$,
$f_{22}\in\mbox{Mat}_{\tau}\left(\mathbb{R}\right)$, $f_{33}\in\mbox{Mat}_{\sigma}$$\left(\mathbb{R}\right)$,
$f_{23},f_{32}^{T}\in\mbox{Mat}_{\tau,\sigma}\left(\mathbb{R}\right)$
and $f_{13},f_{34}^{T},f_{43},f_{31}^{T}\in\mbox{Mat}_{l,\sigma}\left(\mathbb{R}\right)$.
Let $\mathfrak{f}$ be the Lie algebra of $F$. For any quadratic
form, $\mathcal{Q}$, let $\mathfrak{so}\left(\mathcal{Q}\right)$
be the Lie algebra of $G_{\mathcal{Q}}$ and for non negative integers
$z_{1}$ and $z_{2}$ let $\mathfrak{so}\left(z_{1},z_{2}\right)$
be the Lie algebra of $SO\left(z_{1},z_{2}\right)$. Since $\mathfrak{f}$
is a subalgebra of $\mathfrak{so}\left(Q{}_{i_{1},i_{2},i_{3}}\right)$
we have that any $f\in\mathfrak{f}$ must satisfy the relation $f^{T}Q'{}_{i_{1},i_{2},i_{3}}+Q'{}_{i_{1},i_{2},i_{3}}f=0$.
Carrying out this computation yields the following, $f_{14},f_{41}\in\mathfrak{so}\left(l\right)$,
$f_{22}\in\mathfrak{so}\left(\tau_{1},\tau_{2}\right)$, $f_{33}\in\mathfrak{so}\left(\sigma_{1},\sigma_{2}\right)$,
$I_{\tau_{1},\tau_{2}}f_{21}=-f_{42}^{T}$, $I_{\sigma_{1},\sigma_{2}}f_{31}=-f_{43}^{T}$,
$f_{11}=-f_{44}^{T}$, $I_{\sigma_{1},\sigma_{2}}f_{32}=-f_{23}^{T}I_{\tau_{1},\tau_{2}}$,
$f_{12}=-f_{24}^{T}I_{\tau_{1},\tau_{2}}$ and $f_{13}=-f_{34}^{T}I_{\sigma_{1},\sigma_{2}}$.
Considering these relations, define the following subspaces,

$\mathfrak{v}^{+}=\left\{ \left(\begin{array}{cccc}
0 & -v^{T}I_{\tau_{1},\tau_{2}} & 0 & 0\\
0 & 0 & 0 & v\\
0 & 0 & 0 & 0\\
0 & 0 & 0 & 0
\end{array}\right):v\in\mbox{Mat}_{\tau,l}\left(\mathbb{R}\right)\right\} $

$\mathfrak{v}^{-}=\left\{ \left(\begin{array}{cccc}
0 & 0 & 0 & 0\\
v & 0 & 0 & 0\\
0 & 0 & 0 & 0\\
0 & -v^{T}I_{\tau_{1},\tau_{2}} & 0 & 0
\end{array}\right):v\in\mbox{Mat}_{\tau,l}\left(\mathbb{R}\right)\right\} $

$\mathfrak{v}=\left\{ \left(\begin{array}{cccc}
0 & 0 & 0 & 0\\
0 & 0 & v & 0\\
0 & -I_{\sigma_{1},\sigma_{2}}v^{T}I_{\tau_{1},\tau_{2}} & 0 & 0\\
0 & 0 & 0 & 0
\end{array}\right):v\in\mbox{Mat}_{\tau,\sigma}\left(\mathbb{R}\right)\right\} $

$\mathfrak{a}=\left(\begin{array}{cccc}
0 & 0 & 0 & 0\\
0 & \mathfrak{so}\left(\tau_{1},\tau_{2}\right) & 0 & 0\\
0 & 0 & 0 & 0\\
0 & 0 & 0 & 0
\end{array}\right)$

$\mathfrak{d}\left(i_{1},i_{2},i_{3}\right)=\left(\begin{array}{cccc}
0 & 0 & 0 & 0\\
0 & 0 & 0 & 0\\
0 & 0 & \mathfrak{so}\left(\sigma_{1},\sigma_{2}\right) & 0\\
0 & 0 & 0 & 0
\end{array}\right)$

$\mathfrak{c}=\left\{ \left(\begin{array}{cccc}
c & 0 & 0 & 0\\
0 & 0 & 0 & 0\\
0 & 0 & 0 & 0\\
0 & 0 & 0 & -c^{T}
\end{array}\right):c\in\mbox{Mat}_{l}\left(\mathbb{R}\right)\right\} $

$\mathfrak{u}^{-}\left(i_{1},i_{2},i_{3}\right)=\left\{ \left(\begin{array}{cccc}
0 & 0 & 0 & 0\\
0 & 0 & 0 & 0\\
-I_{\sigma_{1},\sigma_{2}}u^{T} & 0 & 0 & 0\\
0 & 0 & u & 0
\end{array}\right):u\in\mbox{Mat}_{l,\sigma}\left(\mathbb{R}\right)\right\} $

$\mathfrak{u}^{+}=\left\{ \left(\begin{array}{cccc}
0 & 0 & u & 0\\
0 & 0 & 0 & 0\\
0 & 0 & 0 & -I_{\sigma_{1},\sigma_{2}}u^{T}\\
0 & 0 & 0 & 0
\end{array}\right):u\in\mbox{Mat}_{l,\sigma}\left(\mathbb{R}\right)\right\} $

$\mathfrak{b}^{+}=\left(\begin{array}{cccc}
0 & 0 & 0 & \mathfrak{so}\left(l\right)\\
0 & 0 & 0 & 0\\
0 & 0 & 0 & 0\\
0 & 0 & 0 & 0
\end{array}\right)$

$\mathfrak{b}^{-}\left(i_{1},i_{2},i_{3}\right)=\left(\begin{array}{cccc}
0 & 0 & 0 & 0\\
0 & 0 & 0 & 0\\
0 & 0 & 0 & 0\\
\mathfrak{so}\left(l\right) & 0 & 0 & 0
\end{array}\right).$

For $0<k\leq\tau$, let

$\mathfrak{v}_{k}=\left\{ \left(\begin{array}{cccc}
0 & 0 & 0 & 0\\
0 & 0 & v & 0\\
0 & -I_{\sigma_{1},\sigma_{2}}v^{T}I_{\tau_{1},\tau_{2}} & 0 & 0\\
0 & 0 & 0 & 0
\end{array}\right):v\in\mbox{Mat}_{\tau,\sigma}\left(\mathbb{R}\right),\: v_{ij}=0\textrm{ for all }i\neq k\right\} $

and for $0<k\leq l$, let

$\mathfrak{u}_{k}^{-}=\left\{ \left(\begin{array}{cccc}
0 & 0 & 0 & 0\\
0 & 0 & 0 & 0\\
-I_{\sigma_{1},\sigma_{2}}u^{T} & 0 & 0 & 0\\
0 & 0 & u & 0
\end{array}\right):u\in\mbox{Mat}_{l,\sigma}\left(\mathbb{R}\right)\: u_{ij}=0\textrm{ for }i\neq k\right\} $

$\mathfrak{u}_{k}^{+}=\left\{ \left(\begin{array}{cccc}
0 & 0 & u & 0\\
0 & 0 & 0 & 0\\
0 & 0 & 0 & -I_{\sigma_{1},\sigma_{2}}u^{T}\\
0 & 0 & 0 & 0
\end{array}\right):u\in\mbox{Mat}_{l,\sigma}\left(\mathbb{R}\right)\: u_{ij}=0\textrm{ for }i\neq k\right\} .$

The dependence on the triple $\left(i_{1},i_{2},i_{3}\right)$ is
only indicated in the cases where it will be necessary later. In order
to simplify the notation, the convention that $\mathfrak{u}^{-}=\mathfrak{u}^{-}\left(i_{1},i_{2},i_{3}\right)$,
$\mathfrak{b}^{-}=\mathfrak{b}^{-}\left(i_{1},i_{2},i_{3}\right)$
and $\mathfrak{d}=\mathfrak{d}\left(i_{1},i_{2},i_{3}\right)$ is
in place. Note that the Lie algebra of $\left(UD\right){}_{i_{1},i_{2},i_{3}}$
is $\mathfrak{b}^{-}\oplus\mathfrak{u}^{-}\oplus\mathfrak{d}$ and
therefore $\mathfrak{u}^{-}\oplus\mathfrak{d}\subseteq\mathfrak{f}$. 

The first step is to show that $\left(\mathfrak{v}\oplus\mathfrak{u}^{+}\right)\cap\mathfrak{f}\neq0$.
Let $\pi_{1}:\mathfrak{f}\rightarrow\mathfrak{u}^{+}$, $\pi_{2}:\mathfrak{f}\rightarrow\mathfrak{v}$,
$\pi_{3}:\mathfrak{f}\rightarrow\mathfrak{b}^{+}$ and $\pi_{4}:\mathfrak{f}\rightarrow\mathfrak{v}^{+}$
be projections. If $\pi_{1}\left(f\right)=\pi_{2}\left(f\right)=\pi_{3}\left(f\right)=\pi_{4}\left(f\right)=0$
for all $f\in\mathfrak{f}$, then $\left\langle x_{1},\dots,x_{l+\tau}\right\rangle $
would be $F$ invariant. Since $\dim\left(\left\langle x_{1},\dots,x_{l+\tau}\right\rangle \right)=l+\tau=m+p'+q'-i_{1}-i_{2}-i_{3}\leq s<d-s$,
the assumption that $F$ has no non trivial invariant subspaces of
dimension less than $d-s$, implies that there exists $f'\in\mathfrak{f}$
such that $\pi_{i}\left(f'\right)\neq0$ for at least one of $1\leq i\leq4$.
Let $\rho:\mathfrak{f}\rightarrow\mathfrak{u}^{-}\oplus\mathfrak{d}$
be a projection. By possibly replacing $f'$ with the element $f'-\left(u+d\right)\in\mathfrak{f}$
for some $u\in\mathfrak{u}^{-}$and $d\in\mathfrak{d}$ we can suppose
that $\rho\left(f'\right)=0$. As before, write $f'$ in block form
and compute the Lie bracket 
\begin{alignat*}{1}
\left[f',\mathfrak{d}\right] & =\left\{ \left(\begin{array}{cccc}
0 & 0 & f_{13}d & 0\\
0 & 0 & f_{23}d & 0\\
0 & -I_{\sigma_{1},\sigma_{2}}d^{T}f_{23}^{T}I_{\tau_{1},\tau_{2}} & 0 & -I_{\sigma_{1},\sigma_{2}}d^{T}f_{13}^{T}\\
0 & 0 & 0 & 0
\end{array}\right):\begin{array}{l}
d\in\mathfrak{so}\left(\sigma_{1},\sigma_{2}\right)\end{array}\right\} \subseteq\left(\mathfrak{v}\oplus\mathfrak{u}^{+}\right)\cap\mathfrak{f}.
\end{alignat*}
Hence, $\left[f',\mathfrak{d}\right]$ is non zero provided that one
of $f_{13}$ or $f_{23}$ is non zero, or equivalently, that $\pi_{1}\left(f'\right)$
or $\pi_{2}\left(f'\right)$ is non zero. Suppose that $\pi_{1}\left(f'\right)=\pi_{2}\left(f'\right)=0$
. By computing the Lie bracket 
\[
\left[f',\mathfrak{u}^{-}\right]=\left\{ \left(\begin{array}{cccc}
0 & 0 & f_{14}u & 0\\
0 & 0 & f_{24}u & 0\\
-I_{\sigma_{1},\sigma_{2}}u^{T}f_{11}^{T} & -I_{\sigma_{1},\sigma_{2}}u^{T}f_{24}^{T}I_{\tau_{1},\tau_{2}} & 0 & -I_{\sigma_{1},\sigma_{2}}u^{T}f_{14}^{T}\\
0 & 0 & f_{11}u & 0
\end{array}\right):\begin{array}{l}
u\in\mbox{Mat}_{l,\sigma}\left(\mathbb{R}\right)\end{array}\right\} 
\]
we see that, so long as at least one of $f_{14}$ or $f_{24}$ is
non zero, or equivalently, $\pi_{3}\left(f'\right)$ or $\pi_{4}\left(f'\right)$
is non zero, there exists $f''\in\left[f',\mathfrak{u}^{-}\right]\subset\mathfrak{f}$,
such that at least one of $\pi_{1}\left(f''\right)$ or $\pi_{2}\left(f''\right)$
is non zero. Thus, we have verified that $\left(\mathfrak{v}\oplus\mathfrak{u}^{+}\right)\cap\mathfrak{f}\neq0$
and hence there is a non trivial subspace contained in $\left(\mathfrak{v}\oplus\mathfrak{u}^{+}\right)\cap\mathfrak{f}$.
Denote this subspace by $\mathfrak{w}$. Since $\left[d,u\right]\in\mathfrak{v}\oplus\mathfrak{u}^{+}$
for all $d\in\mathfrak{d}$ and $u\in\mathfrak{v}\oplus\mathfrak{u}^{+}$
we say that $\mathfrak{v}\oplus\mathfrak{u}^{+}$ is $\mathfrak{\ensuremath{d}}$
invariant for the natural action of $\mathfrak{d}\times\mathfrak{v}\oplus\mathfrak{u}^{+}\rightarrow\mathfrak{v}\oplus\mathfrak{u}^{+}$
given by the Lie bracket. By passing to an irreducible component if
necessary we can assume that $\mathfrak{w}$ is $\mathfrak{d}$ irreducible. 

\textcolor{black}{The next step is to show that $\mathfrak{w}$ is
`diagonally embedded' into $\mathfrak{\mathfrak{\mathfrak{v}}}\oplus\mathfrak{u}^{+}$.
For $1\leq i\leq l$ let $\nu_{i}:\mathfrak{\mathfrak{\mathfrak{v}}}\oplus\mathfrak{u}^{+}\rightarrow\mathfrak{u}_{i}^{+}$
and for $l+1\leq i\leq l+\tau$ let $\nu_{i}:\mathfrak{\mathfrak{\mathfrak{v}}}\oplus\mathfrak{u}^{+}\rightarrow\mathfrak{v}_{i-l}$
be projections. There is a decomposition $\mathfrak{\mathfrak{\mathfrak{v}}}\oplus\mathfrak{u}^{+}=\mathfrak{\mathfrak{v}}_{1}\oplus\ldots\oplus\mathfrak{v}_{\tau}\oplus\mathfrak{u}_{1}^{+}\oplus\ldots\oplus\mathfrak{u}_{l}^{+}$
where it is not possible to split the $\mathfrak{\mathfrak{v}}_{k}$'s
or $\mathfrak{u}_{k}^{+}$'s into further $\mathfrak{d}$ invariant
subalgebra since the action of $\mathfrak{d}$ on $\mathfrak{v}_{k}$
and $\mathfrak{u}_{k}^{+}$ is irreducible for each $k$. Since $\mathfrak{w}$
and $\left\{ \nu_{i}\left(w\right):w\in\mathfrak{w}\right\} $ are
both $\mathfrak{d}$ irreducible subspaces and $\left[d,\nu_{i}\left(w\right)\right]=\nu_{i}\left(\left[d,w\right]\right)$
for all $w\in\mathfrak{w}$ and $d\in\mathfrak{d}$, by Schur's Lemma
for each $1\leq i\leq l+\tau$ either $\mathfrak{w}\cong\left\{ \nu_{i}\left(w\right):w\in\mathfrak{w}\right\} $
or $\nu_{i}\left(w\right)=0$ for all $w\in\mathfrak{w}$. Moreover,
by Schur's Lemma, if for some $1\leq i\leq l+\tau$ we have $\mathfrak{w}\cong\left\{ \nu_{i}\left(w\right):w\in\mathfrak{w}\right\} $
then the isomorphism is given by scalar multiplication. This means
that for all $1\leq i,j\leq l+\tau$ such that $\left\{ \nu_{j}\left(w\right):w\in\mathfrak{w}\right\} \neq0$
there exist constants $c_{i,j}\in\mathbb{R}$ such that $\nu_{i}\left(w\right)=c_{i,j}\nu_{j}\left(w\right)$
for all $w\in\mathfrak{w}$. This is what we mean when we say that
$\mathfrak{w}$ is `diagonally embedded' into $\mathfrak{\mathfrak{\mathfrak{v}}}\oplus\mathfrak{u}^{+}$. }

\textcolor{black}{The next stage of the proof is to show that by conjugating
$\mathfrak{w}$ with a suitable element of $G_{Q_{i_{1},i_{2},i_{3}}}$
we can simplify the embedding of $\mathfrak{w}$ into $\mathfrak{\mathfrak{\mathfrak{v}}}\oplus\mathfrak{u}^{+}$
even further. First we set up some notation. There is an obvious isomorphism,
$\phi:\mathfrak{v}\oplus\mathfrak{u}^{+}\rightarrow\mbox{Mat}_{l+\tau,\sigma}\left(\mathbb{R}\right).$
The fact $\mathfrak{w}$ is diagonally embedded into $\mathfrak{v}\oplus\mathfrak{u}^{+}$
means that $\mathfrak{w}$ can be determined by a vector, $\gamma\in\mathbb{R}^{l+\tau}$
such that 
\[
\phi\left(\mathfrak{w}\right)=\left\{ \left(\begin{array}{l}
\gamma_{1}x\\
\vdots\\
\gamma_{l+\tau}x
\end{array}\right):x\in\mathbb{R}^{\sigma}\right\} .
\]
Let $\phi_{1}$ and $\phi_{2}$ be the corresponding isomorphisms
such that $\phi_{1}:\mathfrak{u}^{+}\rightarrow\mbox{Mat}_{l,\sigma}\left(\mathbb{R}\right)$}
and $\phi_{2}:\mathfrak{v}\rightarrow\mbox{Mat}_{\tau,\sigma}\left(\mathbb{R}\right)$.
Let $a_{1}\in SO\left(l\right)$ and $a_{2}\in O\left(\tau_{1},\tau_{2}\right)$.
Consider the map given by the matrix 
\[
\eta_{1}=\left(\begin{array}{cccc}
a_{1} & 0 & 0 & 0\\
0 & a_{2} & 0 & 0\\
0 & 0 & I_{\sigma} & 0\\
0 & 0 & 0 & a_{1}
\end{array}\right)\in G_{Q_{i_{1},i_{2},i_{3}}},
\]
where in the cases when $l=0$ or $\tau=0$ the map $\eta_{1}$ degenerates
to the obvious one given by $2\times2$ or $3\times3$ block matrices
respectively. Let $\eta_{1}^{-1}\mathfrak{w}\eta_{1}=\mathfrak{w}_{1}$.
One can check that $\eta_{1}^{-1}\left(\mathfrak{v}\oplus\mathfrak{u}^{+}\right)\eta_{1}=\mathfrak{v}\oplus\mathfrak{u}^{+}$
and that 
\[
\phi\left(\mathfrak{w}_{1}\right)=\left(\begin{array}{c}
a_{1}^{-1}\phi_{1}\left(\pi_{1}\left(\mathfrak{w}\right)\right)\\
a_{2}^{-1}\phi_{2}\left(\pi_{2}\left(\mathfrak{w}\right)\right)
\end{array}\right),
\]
where $\pi_{1}$ and $\pi_{2}$ are the projections defined previously.
It is also clear that 
\[
\phi_{1}\left(\pi_{1}\left(\mathfrak{w}\right)\right)=\left\{ \left(\begin{array}{l}
\gamma_{1}x\\
\vdots\\
\gamma_{l}x
\end{array}\right):x\in\mathbb{R}^{\sigma}\right\} \mathrm{\; and\;\;}\phi_{2}\left(\pi_{2}\left(\mathfrak{w}\right)\right)=\left\{ \left(\begin{array}{l}
\gamma_{l+1}x\\
\vdots\\
\gamma_{l+\tau}x
\end{array}\right):x\in\mathbb{R}^{\sigma}\right\} .
\]
Now, $SO\left(l\right)$ acts transitively on the sphere $\sum_{i=1}^{l}x_{i}^{2}=\sum_{i}^{l}\gamma_{i}^{2}$
and $O\left(\tau_{1},\tau_{2}\right)$ acts transitively on the surface
$\sum_{i=1}^{\sigma_{1}}x_{l+\tau+i}^{2}-\sum_{i=\sigma_{1}+1}^{\sigma_{2}}x_{l+\tau+i}^{2}=\sum_{i=1}^{\sigma_{1}}\gamma_{l+\tau+i}^{2}-\sum_{i=\sigma_{1}+1}^{\sigma_{2}}\gamma_{l+\tau+i}^{2}$,
it follows that there exist $a_{1}$ and $a_{2}$ such that $\mathfrak{w}_{1}\subset\mathfrak{u}_{l}^{+}\oplus\mathfrak{\mathfrak{v}}_{1}\oplus\mathfrak{v}_{\tau}$. 

Let $\mu_{1}:\mathfrak{f}\rightarrow\mathfrak{u}_{l}^{+}$ , $\mu_{2}:\mathfrak{f}\rightarrow\mathfrak{\mathfrak{v}}_{1}$
and $\mu_{3}:\mathfrak{f}\rightarrow\mathfrak{\mathfrak{v}}_{\tau}$
be projections. There are now three mutually exclusive cases we must
consider separately. 
\begin{caseenv}
\item There exists $v\in\mathfrak{w}_{1}$ such that $\mu_{1}\left(v\right)\neq0$.
Let $\alpha_{1}$ and $\alpha_{2}$ be real numbers. Because $\mathfrak{w}_{1}$
is diagonally embedded into $\mathfrak{u}_{l}^{+}\oplus\mathfrak{\mathfrak{v}}_{1}\oplus\mathfrak{v}_{\tau}$
we can suppose that $\mu_{2}\left(v\right)=\alpha_{1}\mu_{1}\left(v\right)$
and $\mu_{3}\left(v\right)=\alpha_{2}\mu_{1}\left(v\right)$ for all
$v\in\mathfrak{w}_{1}$. Let $\alpha=\left(\alpha_{2}^{2}-\alpha_{1}^{2}\right)/2$,
consider the map 
\[
\eta_{2}:\begin{cases}
x_{i} & \rightarrow x_{i}\textrm{ for }i\neq l+1,l+\tau\textrm{ or }d\\
x_{l+1} & \rightarrow x_{l+1}+\alpha_{1}x_{l}\\
x_{l+\tau} & \rightarrow x_{l+\tau}+\alpha_{2}x_{l}\\
x_{d} & \rightarrow x_{d}-\alpha_{1}x_{l+1}+\alpha_{2}x_{l+\tau}+\alpha x_{l}.
\end{cases}
\]
Note that $\eta_{2}\in G_{Q_{i_{1},i_{2},i_{3}}}$. By writing elements
of $\mathfrak{w}_{1}$ in the form of linear maps and doing the composition,
one can check that $\eta_{2}^{-1}\mathfrak{w}_{1}\eta_{2}\subseteq\mathfrak{u}_{l}^{+}\oplus\mathfrak{u}_{l}^{-}$.
In a similar manner, one can also check that $\eta_{2}u=u\eta_{2}=u$
for all $u\in\mathfrak{u}^{-}\oplus\mathfrak{d}$. Let 
\[
\mathfrak{c}^{*}=\left\{ \left[u_{1},u_{2}\right]:u_{1}\in\mathfrak{\mathfrak{u}}_{l}^{+},u_{2}\in\mathfrak{u}_{l}^{-}\right\} \cong\mathbb{R}
\]
and note that $\mathfrak{c}^{*}=\left\langle b_{ll}-b_{dd}\right\rangle $,
where $b_{ij}$ denotes the matrix with $1$ in its $i^{\textrm{th}}$
row and $j^{\textrm{th}}$ column and 0 everywhere else. Consider
the map 
\[
\eta_{3}:\begin{cases}
x_{i}\rightarrow x_{i} & \textrm{for }1\leq i<l\textrm{ and }\tau+l<i\leq\sigma+\tau+l\\
x_{l}\rightarrow\frac{1}{\sqrt{2}}\left(x_{\tau+l}-x_{l+\tau+\sigma+1}\right)\\
x_{i}\rightarrow x_{i-1} & \textrm{for }l<i\leq\tau+l\\
x_{i}\rightarrow x_{i+1} & \textrm{for }\sigma+\tau+l<i<d\\
x_{d}\rightarrow\frac{1}{\sqrt{2}}\left(x_{\tau+l}+x_{l+\tau+\sigma+1}\right).
\end{cases}
\]
It is easy to verify that $Q_{i_{1},i_{2},i_{3}}\left(\eta_{3}x\right)=Q_{i_{1},i_{2},i_{3}+1}\left(x\right)$
for all $x\in\mathbb{R}^{d}$. Note that by comparing dimensions we
have 
\[
\mathfrak{u}_{l}^{-}\oplus\mathfrak{d}\oplus\mathfrak{c}^{*}\oplus\mathfrak{u}_{l}^{+}\cong\mathfrak{so}\left(Q_{i_{1},i_{2},i_{3}}|_{\left\langle x_{l},x_{l+\tau+1},\dots,x_{l+\tau+\sigma},x_{d}\right\rangle }\right)
\]
where the isomorphism corresponds to the embedding of $\left\langle x_{l},x_{l+\tau+1},\dots,x_{l+\tau+\sigma},x_{d}\right\rangle $
into $\mathbb{R}^{d}$. Thus 
\begin{alignat*}{1}
\eta_{3}^{-1}\left(\mathfrak{u}_{l}^{-}\oplus\mathfrak{d}\oplus\mathfrak{c}^{*}\oplus\mathfrak{u}_{l}^{+}\right)\eta_{3} & =\mathfrak{d}\left(i_{1},i_{2},i_{3}+1\right).
\end{alignat*}
By construction $\mathfrak{u}^{-}\oplus\mathfrak{d}\oplus\mathfrak{w}\subseteq\mathfrak{f}$
and hence 
\[
\mathfrak{u}_{l}^{-}\oplus\mathfrak{d}\oplus\mathfrak{c}^{*}\oplus\mathfrak{u}_{l}^{+}\subseteq\eta_{2}^{-1}\eta_{1}^{-1}\mathfrak{f}\eta_{1}\eta_{2},
\]
which together with the above formula, implies that 
\[
\mathfrak{d}\left(i_{1},i_{2},i_{3}+1\right)\subseteq\eta_{3}^{-1}\eta_{2}^{-1}\eta_{1}^{-1}\mathfrak{f}\eta_{1}\eta_{2}\eta_{3}.
\]
It is possible to check that $\eta_{3}^{-1}\left(\mathfrak{u}_{l}^{-}\oplus\dots\oplus\mathfrak{u}_{l-1}^{-}\right)\eta_{3}\subseteq\mathfrak{u}^{-}\left(i_{1},i_{2},i_{3}+1\right)$
and thus 
\begin{alignat*}{1}
\left\{ \left[u,d\right]:u\in\eta_{3}^{-1}\left(\mathfrak{u}_{1}^{-}\oplus\dots\oplus\mathfrak{u}_{l-1}^{-}\right)\eta_{3},d\in\mathfrak{d}\left(i_{1},i_{2},i_{3}+1\right)\right\}  & =\mathfrak{u}^{-}\left(i_{1},i_{2},i_{3}+1\right)\\
 & \subseteq\eta_{3}^{-1}\eta_{3}^{-1}\eta_{2}^{-1}\eta_{1}^{-1}\mathfrak{f}\eta_{1}\eta_{2}\eta_{3}.
\end{alignat*}
Moreover 
\begin{alignat*}{1}
\left\{ \left[u,v\right]:u\in\mathfrak{u}^{-}\left(i_{1},i_{2},i_{3}+1\right),v\in\mathfrak{u}^{-}\left(i_{1},i_{2},i_{3}+1\right)\right\}  & =\mathfrak{b}^{-}\left(i_{1},i_{2},i_{3}+1\right)\\
 & \subseteq\eta_{3}^{-1}\eta_{3}^{-1}\eta_{2}^{-1}\eta_{1}^{-1}\mathfrak{f}\eta_{1}\eta_{2}\eta_{3}.
\end{alignat*}
Therefore, $\left(UD\right){}_{i_{1},i_{2},i_{3}+1}\leq\eta_{3}^{-1}\eta_{2}^{-1}\eta_{1}^{-1}F\eta_{1}\eta_{2}\eta_{3}\leq G_{Q_{i_{1},i_{2},i_{3+1}}}$,
which is the third conclusion of the Lemma. Note that if $i_{3}=m$
it is impossible that there exists $v\in\mathfrak{w}_{1}$ so that
$\mu_{1}\left(v\right)\neq0$ and this case cannot occur if $i_{3}=m$. 
\item For all $v\in\mathfrak{w}_{1}$, $\mu_{1}\left(v\right)=0$ and either,
$\mu_{2}\left(v\right)=0$ for all $v\in\mathfrak{w}_{1}$ or $\mu_{3}\left(v\right)=0$
for all $v\in\mathfrak{w}_{1}$. In other words, either $\mathfrak{w}_{1}=\mathfrak{v}_{1}$
or $\mathfrak{w}_{1}=\mathfrak{v}_{\tau}$. Both cases can be treated
in an identical fashion, so suppose that $\mathfrak{w}_{1}=\mathfrak{v}_{1}$.
Moreover, suppose that $\tau_{1}>0$ or else, by possibly modifying
$a_{2}\in O\left(\tau_{1},\tau_{2}\right)$, we could suppose that
$\mathfrak{w}_{1}=\mathfrak{v}_{\tau}$ and $\tau_{2}>0$ and the
arguments would be the same up to minor modifications of the map involved.
Consider the map 
\[
\eta_{4}:\begin{array}{l}
\begin{cases}
x_{l+1} & \rightarrow x_{l+\tau}\\
x_{i} & \rightarrow x_{i}\quad\quad\textrm{ for }1\leq i\leq l\textrm{ or }l+\tau<i\leq d\\
x_{i} & \rightarrow x_{i+1}\quad\textrm{ for }l+1<i\leq l+\tau
\end{cases}\end{array}
\]
with the properties that $Q_{i_{1},i_{2},i_{3}}\left(\eta_{4}x\right)=Q_{i_{1}+1,i_{2},i_{3}}\left(x\right)$
for all $x\in\mathbb{R}^{d}$. By comparing dimensions 
\[
\mathfrak{\mathfrak{d}}\oplus\mathfrak{v}_{1}\cong\mathfrak{so}\left(Q_{i_{1},i_{2},i_{3}}|_{\left\langle x_{l+1},x_{l+\tau+1},\dots,x_{l+\tau+\sigma}\right\rangle }\right)
\]
where the isomorphism corresponds to the embedding of $\left\langle x_{l+1},x_{l+\tau+1},\dots,x_{l+\tau+\sigma}\right\rangle $
into $\mathbb{R}^{d}$. Thus 
\begin{alignat*}{1}
\eta_{4}^{-1}\left(\mathfrak{\mathfrak{d}}\oplus\mathfrak{v}_{1}\right)\eta_{4}=\mathfrak{d}\left(i_{1}+1,i_{2},i_{3}\right) & \subseteq\eta_{4}^{-1}\eta_{1}^{-1}\mathfrak{f}\eta_{1}\eta_{4}.
\end{alignat*}
It is possible to check that $\eta_{4}^{-1}\mathfrak{u}^{-}\eta_{4}\subset\mathfrak{u}^{-}\left(i_{1}+1,i_{2},i_{3}\right)$
and 
\begin{alignat*}{1}
\left\{ \left[u,d\right]:u\in\eta_{4}^{-1}\mathfrak{u}^{-}\eta_{4},d\in\mathfrak{d}\left(i_{1}+1,i_{2},i_{3}\right)\right\}  & =\mathfrak{u}^{-}\left(i_{1}+1,i_{2},i_{3}\right)\\
 & \subseteq\eta_{4}^{-1}\eta_{1}^{-1}\mathfrak{f}\eta_{1}\eta_{4}.
\end{alignat*}
Moreover 
\[
\mathfrak{b}^{-}\left(i_{1},i_{2},i_{3}\right)=\mathfrak{b}^{-}\left(i_{1}+1,i_{2},i_{3}\right)\subseteq\eta_{4}^{-1}\eta_{1}^{-1}\mathfrak{f}\eta_{1}\eta_{4}.
\]
Therefore, provided that $\tau_{1}>0$ we have $\left(UD\right){}_{i_{1}+1,i_{2},m}\leq\eta_{4}^{-1}\eta_{1}^{-1}F\eta_{1}\eta_{4}\leq G_{Q_{i_{1}+1,i_{2},m}}$,
which is the first conclusion of the Lemma. If the argument was repeated
for the opposite case, when $\mathfrak{w}_{1}=\mathfrak{v}_{\tau}$
and $\tau_{2}>0$, we would obtain $\left(UD\right){}_{i_{1},i_{2}+1,m}\leq\eta_{4}^{-1}\eta_{1}^{-1}F\eta_{1}\eta_{4}\leq G_{Q_{i_{1},i_{2}+1,m}}$,
which is the second conclusion of the Lemma. Note that if $i_{1}=p'$
then $\tau_{1}=0$ and hence the second conclusion can occur but not
the first. Similarly, if $i_{2}=q'$ then $\tau_{2}=0$ and $\tau_{1}>0$
and the first conclusion can occur but not the second.
\item For all $v\in\mathfrak{w}_{1}$, $\mu_{1}\left(v\right)=0$ and there
exists $c\in\mathbb{R}\setminus\left\{ 0\right\} $ such that $\mu_{2}\left(v\right)=c\mu_{3}\left(v\right)$
for all $v\in\mathfrak{w}_{1}$. Note that $O\left(\tau_{1},\tau_{2}\right)$
acts transitively on the hyperbola $x_{l+1}^{2}-x_{l+\tau}^{2}=c^{2}-1$.
If $c^{2}\neq1$, the hyperbola contains a vector such that the projection
onto the second co-ordinate is 0, hence we can modify $a_{2}$ so
that $\mathfrak{w}_{1}\subseteq\mathfrak{v}_{1}$ and then repeat
the arguments used for case 2. Therefore, we can assume that $c^{2}=1$.
Moreover, because $O\left(\tau_{1},\tau_{2}\right)$ contains the
transformations $x_{l+1}\rightarrow-x_{l+1}$ we may suppose that
$c=-1$. Consider the map 
\[
\eta_{5}:\begin{cases}
x_{l+1} & \rightarrow\frac{1}{\sqrt{2}}\left(x_{l+1}+x_{d}\right)\\
x_{l+\tau} & \rightarrow\frac{1}{\sqrt{2}}\left(x_{l+1}-x_{d}\right)\\
x_{i} & \rightarrow x_{i}\textrm{ for }1\leq i<l+1\textrm{ or }l+1<i<l+\tau\\
x_{i} & \rightarrow x_{i-1}\textrm{ for }l+\tau<i\leq d
\end{cases}
\]
with the property that $Q_{i_{1},i_{2},i_{3}}\left(\eta_{5}x\right)=Q_{i_{1}+1,i_{2}+1,i_{3}-1}\left(x\right)$
for all $x\in\mathbb{R}^{d}$. Since $\mathfrak{f}\subset\mathfrak{so}\left(Q_{i_{1},i_{2},i_{3}}\right)$
and by construction $\mathfrak{w}\oplus\mathfrak{\mathfrak{u}^{-}\oplus\mathfrak{d}}\subseteq\mathfrak{f}$,
the fact that $\eta_{1}\in G_{Q_{i_{1},i_{2},i_{3}}}$, means that
\begin{alignat*}{1}
\eta_{5}^{-1}\eta_{1}^{-1}\left(\mathfrak{w}\oplus\mathfrak{u}^{-}\oplus\mathfrak{d}\right)\eta_{1}\eta_{5} & =\left\{ \left(\begin{array}{ccccc}
0 & 0 & 0 & 0 & 0\\
0 & 0 & 0 & 0 & 0\\
-I_{\sigma_{1},\sigma_{2}}u^{T} & -I_{\sigma_{1},\sigma_{2}}v^{T} & 0 & d & 0\\
0 & 0 & 0 & u & 0\\
0 & 0 & 0 & v & 0
\end{array}\right):\begin{array}{l}
d\in\mathfrak{so}\left(I_{\sigma_{1},\sigma_{2}}\right)\\
v\in\mathbb{R}^{\sigma}\\
u\in\mbox{Mat}_{l,\sigma}\left(\mathbb{R}\right)
\end{array}\right\} \\
 & =\mathfrak{u}^{-}\left(i_{1}+1,i_{2}+1,i_{3}-1\right)\oplus\mathfrak{d}\left(i_{1}+1,i_{2}+1,i_{3}-1\right)\\
 & \subseteq\eta_{5}^{-1}\eta_{1}^{-1}\mathfrak{f}\eta_{1}\eta_{5}\\
 & \subseteq\mathfrak{so}\left(Q_{i_{1}+1,i_{2}+1,i_{3}-1}\right).
\end{alignat*}
The matrix in the previous calculation can be explained as follows.
There is a decomposition of $\mathfrak{so}\left(Q_{i_{1}+1,i_{2}+1,i_{3}-1}\right)$
into subspaces as carried out at the start of the proof. The bottom
two rows and the left most two columns form $\mathfrak{u}^{-}\left(i_{1}+1,i_{2}+1,i_{3}-1\right)$
and should be considered as one row and one column respectively in
the decomposition. The rest of the blocks are the sizes that would
be obtained from the decomposition. Moreover 
\begin{alignat*}{1}
\left\{ \left[u,v\right]:u\in\mathfrak{u}^{-}\left(i_{1}+1,i_{2}+1,i_{3}-1\right),v\in\mathfrak{u}^{-}\left(i_{1}+1,i_{2}+1,i_{3}-1\right)\right\}  & =\mathfrak{b}^{-}\left(i_{1}+1,i_{2}+1,i_{3}-1\right)\\
 & \subseteq\eta_{5}^{-1}\eta_{1}^{-1}\mathfrak{f}\eta_{1}\eta_{5}.
\end{alignat*}
Therefore, \foreignlanguage{english}{$\left(UD\right){}_{i_{1}+1,i_{2}+1,i_{3}-1}\leq\eta_{5}^{-1}\eta_{1}^{-1}F\eta_{1}\eta_{5}\leq G_{Q_{i_{1}+1,i_{2}+1,i_{3}-1}}$},
which is the fourth conclusion of the Lemma. Note that if $i_{1}=p'$
or $i_{2}=q'$ then this case will not arise because either $\tau_{1}$
or $\tau_{2}$ will be zero. Similarly, if $i_{3}=-\min\left\{ p',q'\right\} ,$
then either $i_{1}=p'$ or $i_{2}=q'$, hence this case cannot arise
if $i_{3}=-\min\left\{ p',q'\right\} $.
\end{caseenv}
Since the three cases considered above are mutually exclusive and
there are no other possibilities, this completes the proof the Lemma.
\end{proof}
We can now complete the proof of Proposition \foreignlanguage{english}{\prettyref{cor:hopefully teh eend}.}
\begin{proof}[Proof of Proposition \foreignlanguage{english}{\prettyref{cor:hopefully teh eend}}]
Let $F$ be as in the statement of the Proposition, then its clear
that $H_{0}^{*}\leq g_{d}Fg_{d}^{-1}\leq G_{0}$. Recall from Lemma
\ref{canonical} that 
\[
Q_{0}\left(x\right)=Q_{m+1,\ldots,s}\left(x\right)+2\sum_{i=1}^{m}x_{i}x_{s+r+n+i}+\sum_{i=s+1}^{s+r}x_{i}^{2}-\sum_{i=s+r+1}^{s+r+n}x_{i}^{2}.
\]
There exists $\eta_{0}\in GL_{d}\left(\mathbb{R}\right)$ such that
$\eta_{0}:x_{i}\rightarrow x_{i}$ for $1\leq i\leq m$ and $s+1\leq i\leq d$
and $Q_{m+1,\ldots,s}\left(\eta_{0}x\right)=\sum_{i=m+1}^{p'}x_{i}^{2}-\sum_{i=m+p'+1}^{p'+q'}x_{i}^{2}$,
therefore $\eta_{0}G_{0}\eta_{0}^{-1}=G_{Q_{0,0,0}}$ and $\eta_{0}H_{0}^{*}\eta_{0}^{-1}=H_{0}^{*}$
and hence $H_{0}^{*}\leq\eta_{0}g_{d}Fg_{d}^{-1}\eta_{0}^{-1}\leq G_{Q_{0,0,0}}$.
Since $H_{0}^{*}=\left(UD\right){}_{0,0,0}$ and by the assumptions
on $F$, $\eta_{0}g_{d}Fg_{d}^{-1}\eta_{0}^{-1}$ has no non trivial
invariant subspaces of dimension less than $d-s$, we can apply Lemma
\ref{lem:f=00003Dso(p,q) general} to get that there exists $\eta\in GL_{d}\left(\mathbb{R}\right)$
such that either:
\begin{enumerate}
\item $\left(UD\right){}_{1,0,0}\leq\eta\eta_{0}g_{d}Fg_{d}^{-1}\eta_{0}^{-1}\eta^{-1}\leq G_{Q_{1,0,0}}$,
\item $\left(UD\right){}_{0,1,0}\leq\eta\eta_{0}g_{d}Fg_{d}^{-1}\eta_{0}^{-1}\eta^{-1}\leq G_{Q_{0,1,0}}$,
\item $\left(UD\right){}_{0,0,1}\leq\eta\eta_{0}g_{d}Fg_{d}^{-1}\eta_{0}^{-1}\eta^{-1}\leq G_{Q_{0,0,1}}$,
\selectlanguage{english}%
\item $\left(UD\right){}_{1,1,-1}\leq\eta\eta_{0}g_{d}Fg_{d}^{-1}\eta_{0}^{-1}\eta^{-1}\leq G_{Q_{1,1,-1}}.$
\end{enumerate}
It is clear that $\eta\eta_{0}g_{d}Fg_{d}^{-1}\eta_{0}^{-1}\eta^{-1}$
has no non trivial invariant subspaces of dimension less than $d-s$.
Therefore, we can apply Lemma \ref{lem:f=00003Dso(p,q) general} repeatedly.
Since at each stage we always increase at least one of the indices
and we can only repeat the fourth case at most $\min\left\{ p',q'\right\} $
times, its clear that after repeating the process a finite number
of times we will obtain the inclusion $SO\left(p,q\right)\leq gg_{d}Fg_{d}^{-1}g^{-1}\leq G_{Q_{p',q',m}}$
for some $g\in GL_{d}\left(\mathbb{R}\right)$ and this implies that
$F=G_{Q}.$
\end{proof}
Finally, we are ready to prove Theorem \ref{conj:multilinear}. 
\begin{proof}[\selectlanguage{english}%
Proof of Theorem 1.2\selectlanguage{british}%
]
Rewrite $\overline{M\left(X_{\mathbb{Z}}\right)}=\overline{\left\{ M\left(x\right):x\in X_{\mathbb{Z}}\right\} }$.
Since $M$ is $H_{Q,M}^{*}$ invariant, we see 
\[
\overline{\left\{ M\left(x\right):x\in X_{\mathbb{Z}}\right\} }=\overline{\left\{ M\left(H_{Q,M}^{*}x\right):x\in X_{\mathbb{Z}}\right\} }.
\]
Now, since $X_{\mathbb{Z}}$ is $\Gamma_{Q}$ invariant 
\[
\overline{\left\{ M\left(H_{Q,M}^{*}x\right):x\in X_{\mathbb{Z}}\right\} }\supseteq\overline{\left\{ M\left(H_{Q,M}^{*}\Gamma_{Q}x\right):x\in X_{\mathbb{Z}}\right\} }.
\]
By Ratner's Theorem (Theorem \ref{thm:(Ranghunthans-togological-conjec})
we have 
\[
\overline{\left\{ M\left(H_{Q,M}^{*}\Gamma_{Q}x\right):x\in X_{\mathbb{Z}}\right\} }\supseteq\left\{ M\left(Fx\right):x\in X_{\mathbb{Z}}\right\} 
\]
for some closed connected subgroup $F$, such that $H_{Q,M}^{*}\leq F\leq G_{Q}$.
Proposition 3.2 in \cite{MR1092178} says that $F$ is the connected
component containing the identity of the real points of an algebraic
group defined over $\mathbb{Q}$. In particular, by Theorem 7.7 of
\cite{MR1278263} this implies $\overline{F\cap SL_{d}\left(\mathbb{Q}\right)}=F$
and so we can apply Lemma \ref{lem:no invariant subspaces} to see
that $F$ has no invariant subspaces of dimension less than $d-s$,
hence we can apply Proposition \ref{cor:hopefully teh eend} to get
that $F=G_{Q}$. Since $G_{Q}$, being the identity component of $SO\left(p,q\right)$,
acts transitively on connected components of $X_{\mathbb{R}}$ we
have $\left\{ M\left(G_{Q}x\right):x\in X_{\mathbb{Z}}\right\} =\left\{ M\left(x\right):x\in X_{\mathbb{R}}\right\} $
since if $X_{\mathbb{R}}$ is not connected, then if $x\in X_{\mathbb{Z}}$,
we have $-x\in X_{\mathbb{Z}}$ and $x$ and $-x$ lie in the two
separate components of $X_{\mathbb{R}}$. As remarked in the introduction,
the fact that $Q|_{\textrm{Ker}\left(M\right)}$ is indefinite implies
that $X_{\mathbb{R}}\cap\left\{ x\in\mathbb{R}^{d}:M\left(x\right)=b\right\} $
is non compact for every $b\in\mathbb{R}^{s}$, in particular this
implies that $X_{\mathbb{R}}\cap\left\{ x\in\mathbb{R}^{d}:M\left(x\right)=b\right\} $
is non empty for every $b\in\mathbb{R}^{s}$ or, in other words, that
$\left\{ M\left(x\right):x\in X_{\mathbb{R}}\right\} =\mathbb{R}^{s}$
and so we are done.
\end{proof}

\section*{Acknowledgements}

I would like to thank my supervisor, Alex Gorodnik, for suggesting
the problem to me and for many helpful discussions and suggestions,
generally ensuring accuracy of the stated results. I would also like
to thank the referee who reviewed an earlier version of this paper
for pointing out many inaccuracies and providing a clear and comprehensive
report. 

\bibliographystyle{amsalpha}
\bibliography{References}

\end{document}